\documentclass[11pt,a4paper]{amsart}
\usepackage[utf8]{inputenc}
\usepackage[english]{babel}
\usepackage{amsmath}
\usepackage{amsfonts}
\usepackage{amssymb}
\usepackage{amsthm}
\usepackage{float}

\usepackage[colorinlistoftodos, textwidth=3cm]{todonotes}
\newcounter{cccomment}

\newcounter{dccomment}

\usepackage[colorlinks=true]{hyperref}

\usepackage{latexsym}    
\usepackage{amssymb}   
\usepackage{amsmath}    
\usepackage{amsbsy}
\usepackage{amsthm}
\usepackage{amsgen}
\usepackage{amsfonts}
\usepackage{array}
\usepackage{lmodern}
\usepackage[all]{xy}    

\usepackage{graphicx}
\usepackage{caption}
\usepackage[labelformat=simple]{subcaption}

\newsavebox{\largestimage}

\usepackage{tikz}
\usepackage{tikz-cd}
\usetikzlibrary{decorations.pathmorphing}
\usetikzlibrary{calc, decorations.markings}

\usepackage{color}
\usepackage{verbatim}
\usepackage{url}
\usepackage{enumerate}
\numberwithin{figure}{section}

\tikzset{snake it/.style={decorate, decoration=snake}}

\usepackage{lineno}

\DeclareMathOperator{\Diff}{Diff}

\DeclareMathOperator{\hol}{hol}

\newtheorem{theorem}{Theorem}  
\newtheorem{corollary}[theorem]{Corollary}  
\newtheorem{conj}[theorem]{Conjecture}		
		
		\newtheorem{thm}{Theorem}[section]
		
		\newtheorem{lem}[thm]{Lemma}

		\newtheorem{prop}[thm]{Proposition}
	\theoremstyle{definition}	
		\newtheorem{remark}[thm]{Remark}
		
	\theoremstyle{definition}
	\newtheorem{example}[thm]{Example}
	\newtheorem{setup}[thm]{Setup}

 \newtheoremstyle{TheoremNum}
        {\topsep}{\topsep}              
        {\itshape}                      
        {}                              
        {\bfseries}                     
        {.}                             
        { }                             
        {\thmname{#1}\thmnote{ \bfseries #3}}
    \theoremstyle{TheoremNum}

\numberwithin{equation}{section}

\newtheorem*{ack}{Acknowledgments}

\usepackage{hyperref}
\hypersetup{
    colorlinks=true,
    linkcolor=blue,
    citecolor=blue,
    urlcolor=blue,
    pdfauthor={Caterina Campagnolo, Diego Corro},
    pdftitle={Some Title}
}

\newcommand{\C}{\mathbb{C}} 
\newcommand{\Cl}{\mathrm{cl}} 
\newcommand{\fol}{\mathcal{F}} 
\newcommand{\GL}{\mathrm{GL}} 
\newcommand{\Hol}{\mathrm{Hol}} 
\newcommand{\N}{\mathbb{N}}
\newcommand{\R}{\mathbb{R}} 
\newcommand{\Sp}{\mathbb{S}} 

\newcommand{\Tr}{\mathcal{T}}
\newcommand{\Z}{\mathbb{Z}} 

\usepackage{mathtools,xparse}

\DeclarePairedDelimiter{\oldnormaux}{\bracevert}{\bracevert}

\NewDocumentCommand{\oldnorm}{som}{%
  \IfBooleanTF{#1}
    {\oldnormaux*{#3}}
    {\IfNoValueTF{#2}
       {\oldnormaux*{\vphantom{dq}#3}}
       {\oldnormaux[#2]{#3}}%
    }%
}

\begin{document}

\author[C.~CAMPAGNOLO]{Caterina Campagnolo}
\address[C.~CAMPAGNOLO]{Institut f\"ur Algebra und Geometrie, Karlsruher Institut f\"ur Technologie (KIT), Karlsruhe, Germany.}
\email{\href{mailto:caterina.campagnolo@kit.edu}
{caterina.campagnolo@kit.edu}}
\urladdr{\url{http://www.math.kit.edu/iag7/~campagnolo/de}}


		
\author[D.~Corro]{Diego Corro}
\address[D.~CORRO]{Institut f\"ur Algebra und Geometrie, Karlsruher Institut f\"ur Technologie (KIT), Karlsruhe, Germany.}
\email{\href{mailto:diego.corro@kit.edu}
{diego.corro@kit.edu}}
\urladdr{\url{http://www.math.kit.edu/iag5/~corro/en}}



\title[Foliated simplicial volume and circle foliations]{Integral foliated simplicial volume and circle foliations}



\setlength{\overfullrule}{5pt}
	\begin{abstract}
		We show that the integral foliated simplicial volume  of a connected compact oriented smooth manifold with a regular foliation by circles vanishes.
	\end{abstract}
	
\maketitle	


\section{Introduction}

In his proof of Mostow rigidity  in \cite{Gromov1982}, Gromov introduced the concept of simplicial volume for an oriented compact connected topological manifold $M$. This is a homotopy  invariant of $M$, which measures the complexity of singular fundamental cycles of $M$ with $\R$-coefficients. When $M$ is a smooth manifold, its simplicial volume also has a geometric interpretation as an obstruction to the existence of a complete Riemannian metric of negative sectional curvature on the given manifold (see \cite{Gromov1982,Thurston1997, InoueYano1982}). Furthermore Gromov showed that the simplicial volume is dominated by the minimal volume of the manifold  (see \cite{Gromov1982}). In this light, the simplicial volume is an obstruction to the existence of a sequence of collapsing Riemannian metrics on the given manifold (see \cite{Pansu1983}).

A long-standing  conjecture about this interplay between topology and geometry was stated by Gromov \cite[$8. A_4 (+)$]{Gromov1993}, \cite[3.1 (e) on p. 769]{Gromov2009}:

\begin{conj}[]\label{Con: Simplicial volume = 0 imlies E(M) = 0}
Let $M$ be an oriented closed (compact without boundary) connected aspherical manifold. If the simplicial volume of $M$ vanishes, then  the $L^2$-Betti numbers of $M$ vanish. In particular, the Euler characteristic of $M$ also vanishes.
\end{conj}

An approach to give a positive answer to Conjecture~\ref{Con: Simplicial volume = 0 imlies E(M) = 0} is laid out in \cite{Gromov1999, Schmidt2005}: Consider a suitable  integral approximation of the simplicial volume, and then apply a Poincar\'{e} duality theorem to bound the $L^2$-Betti numbers, in terms of this integral approximation. Finally,  relate the integral approximation to the simplicial volume on aspherical manifolds. An instance of such an approximation is the integral foliated simplicial volume (see \cite{Gromov1999, Schmidt2005}), which we consider here. 

In this note we show that in the presence of a circle foliation and a big fundamental group, the  integral foliated simplicial volume vanishes:

\begin{theorem}\label{Theorem: circle foliation implies the vanishing of foliated simplicial volume}
Let $M$ be a compact oriented smooth manifold equipped with a regular smooth circle foliation $\fol$, such that the inclusion of each  leaf of $\fol$ is $\pi_1$-injective and has finite holonomy. Then the (relative) integral foliated simplicial volume of $M$ vanishes:
\[
	\oldnorm{M,\partial M} = 0.
\]
\end{theorem}



The work of Schmidt in \cite{Schmidt2005} provides the following corollary to Theorem~\ref{Theorem: circle foliation implies the vanishing of foliated simplicial volume}.

\begin{corollary}
Let $(M,\fol)$ be a closed connected oriented aspherical smooth manifold, and $\fol$ a regular smooth circle foliation, such that the inclusion of each leaf is $\pi_1$-injective and has finite holonomy. Then the $L^2$-Betti numbers of $M$ vanish. In particular $\chi(M)=0$.
\end{corollary}
The work of Löh in \cite{Loeh2019} provides another corollary to Theorem~\ref{Theorem: circle foliation implies the vanishing of foliated simplicial volume}, involving a different invariant, namely the cost. The cost is a randomized version of the minimal number of generators of a group.
\begin{corollary}
Let $(M,\fol)$ be a closed connected oriented aspherical smooth manifold, and $\fol$ a regular smooth circle foliation, such that the inclusion of each leaf is $\pi_1$-injective and has finite holonomy. Then $\mathrm{Cost}(\pi_1(M))=1$, that is the manifold $M$ is cheap.
\end{corollary}
%


Given a smooth manifold $M$, we may consider different notions of  ``symmetry" for $M$. A general notion is to  split $M$ into a family of submanifolds, that retain certain desired geometric or topological properties. In \cite{CheegerGromov1986}, Cheeger and Gromov studied a type of decomposition called an $F$-structure. This is a direct generalization of a smooth torus action on a smooth manifold. In particular, they show that if a Riemannian manifold $M$ admits a polarized  $F$-structure of positive dimension, then the minimal volume of $M$ vanishes. Since the minimal volume dominates the simplicial volume, the simplicial volume of $M$ also vanishes. Furthermore, in this case the Euler characteristic of $M$ is also zero (see \cite[Proposition~1.5]{CheegerGromov1986} \cite[p.~75]{Pansu1983}).


In the particular case when the splitting of $M$ is given by the orbits of a smooth circle action, the Corollary to the Vanishing Theorem \cite[p. 41]{Gromov1982} implies that the simplicial volume of $M$ vanishes. This fact was independently proven by Yano in \cite{Yano1982}. His proof is quite geometric and relies heavily on the stratification of the orbit space by orbit type. This proof has been extended to the setting of the integral foliated simplicial volume by Fauser in \cite{Fauser2019}, when one adds the assumption that the inclusion of every orbit in $M$ is injective at the level of fundamental groups. 
We point out that this condition is not satisfied when the action has fixed points.
A more general notion of symmetry can be found in the context of foliations.  They arise naturally as solutions to differential equations (see \cite[Chapter~1]{Moerdijk}). Moreover, smooth group actions, as well as the fibers of a smooth fiber bundle, are some examples of  smooth foliations.

We follow the approach of \cite{Fauser2019,Yano1982} to prove Theorem~\ref{Theorem: circle foliation implies the vanishing of foliated simplicial volume}. The hypothesis requiring the holonomy groups to be finite is necessary to obtain the stratification of the leaf space. Indeed,  Sullivan in \cite{Sullivan1976} constructed a $5$-dimensional smooth manifold with a foliation by circles, with leaves having infinite holonomy. For this particular example the leaf space does not admit an orbifold structure.  Furthermore, the condition that  for any  leaf its inclusion into the manifold induces an injective map between the fundamental groups is also necessary for our proof. Simple examples of a regular foliation by circles where the inclusion is not $\pi_1$-injective are given by the Hopf fibrations $\Sp^{2n+1}\to \C P^{n}$.  

Any smooth circle action gives rise to a circle foliation, but there are many instances of circle foliations not coming from circle actions. A simple example is given by considering non-orientable circle bundles over a non-orientable base. For example, the unit tangent bundle of a closed non-orientable surface  is a circle foliation not coming from a circle action. This  holds since the unit tangent bundle is not a principal circle bundle: if it were, then the structure group of the tangent bundle would be $SO(2)$, but since the surface is non-orientable the structure group of the tangent bundle cannot be reduced to $SO(2)$. Also observe that the total space of the tangent bundle of any manifold is an orientable manifold. In particular the tangent bundle of the  Möbius band $\{(\theta,y)\in [0,1]\times [-1,1]\}/(0,y)\sim (1,-y)$ is diffeomorphic to $\Sp^1\times [-1,1]\times\R^2$. Another concrete example is given by the unit tangent bundle of the Klein bottle. We remark that for these examples, there is a finite cover of the manifold with a smooth circle action such that the orbits of the action cover the leaves of the foliation. Other non-homogeneous examples are given by non-orientable Seifert fibrations over a non-orientable surface (see \cite[Section~5.2]{Orlik}). We point out that the universal cover of a Seifert fibration is one of $\Sp^3$, $\Sp^2\times \R$ and $\R^3$ (see \cite[Proposition~1]{Preaux2014}), and they all admit a smooth circle action. We remark that, for these cases, the conclusion of the main theorem can also be obtained by the work of Fauser, Friedl, Löh in \cite[Theorem 1.7]{FauserFriedlLoeh2019}.

To prove Theorem  \ref{Theorem: circle foliation implies the vanishing of foliated simplicial volume}, we need to check that Yano's and Fauser's techniques from \cite{Fauser2019,Yano1982} apply in the case of foliations instead of circle actions. After the hollowing procedure, we are left with a certain circle bundle $M_{n-2}\to M_{n-2}/\fol_{n-2}$. If it is an orientable bundle, our proof is a straightforward translation of the proof in \cite{Fauser2019}. But the bundle can be non-orientable over a non-orientable base (thus still having orientable total space): this is a novelty of foliations with respect to circle actions. Then we need an additional argument that exploits the flexibility of the integral foliated simplicial volume (see Section \ref{S: Proof}, in particular Setup \ref{Setup: restrictions of representations}).

In general, regular foliations by circles are instances of singular circle fiberings over polyhedra, introduced by Edmonds and Fintushel in \cite{EdmondsFintushel1976}. They proved that a singular circle fibering on a smooth manifold is given by a group action if and only if the bundle part of the fibering is orientable (see \cite[Theorem~3.8]{EdmondsFintushel1976}). To the best of our knowledge, it is not known whether an arbitrary smooth manifold with a foliation by circles with finite holonomy admits a finite cover with a group action. If so, our main result would follow by multiplicativity of the integral foliated simplicial volume and the main result in \cite{Fauser2019}.

The condition of $\pi_1$-injectivity is a technical one: it is used so that we have only a global representation of $\pi_1(M)$ over a fixed essentially free standard Borel space to consider, instead of a potentially different one for every leaf. This technicality might be overcome by constructing a suitable family of representations and essentially free standard Borel spaces at each step of the hollowing construction (the hollowing construction is presented in Section \ref{S: Hollowings}).


We organize the present note as follows: in Section~\ref{S: Preliminaries} we present the preliminaries as well as a series of clarifying examples. In Section~\ref{S: Proof} we prove Theorem~\ref{Theorem: circle foliation implies the vanishing of foliated simplicial volume}.

\vspace*{1em}	
	
\begin{ack}
We thank Roman Sauer for pointing us in the direction of studying the integral foliated simplicial volume for foliated manifolds. We also thank Michelle Bucher, Daniel Fauser, and Fernando Galaz-Garcia for useful conversations during the preparation of the present manuscript, and Clara Löh for very useful comments about the non-orientable case. 
The authors acknowledge funding by the Deutsche Forschungsgemeinschaft (DFG, German Research Foundation) – 281869850 (RTG 2229).
\end{ack}	

\section{Preliminaries}\label{S: Preliminaries}

\subsection{Simplicial volume}

We begin by defining the \emph{simplicial volume}, also known as the \emph{Gromov invariant} or \emph{Gromov norm}, of a topological compact connected oriented manifold $M$,  with possibly empty boundary $\partial M$. Given a  singular $k$-chain $z=\sum_ia_i\sigma_i\in C_k(M, \mathbb{R})$, we define its \emph{$\ell_1$-norm} as
\[\|z\|_1=\sum_i |a_i| \,,\]
where $\sigma_i\colon \Delta^k\rightarrow M$ is a singular simplex of dimension $k$.

The \emph{simplicial volume} of $M$ is the infimum over the $\ell_1$-norms of the (relative) real cycles representing the fundamental class:
\[\|M, \partial M\|=\inf\left\{\sum_i |a_i|\,\vline\, [M, \partial M]=\left[\sum_i a_i\sigma_i\right]\in H_n(M, \partial M; \mathbb{R})\right\}.\]


\subsection{Integral foliated simplicial volume}

Consider a topological connected manifold $M$ and its universal cover $\widetilde{M}$. Denote by  $\Gamma=\pi_1(M)$ the fundamental group of $M$. Then $\Gamma$ acts on $\widetilde{M}$ by deck transformations. This induces a natural action of $\Gamma$ on $C_k(\widetilde{M}, \mathbb{Z})$. There is a natural identification
\[C_k(M, \mathbb{Z})\cong \mathbb{Z}\otimes_{\mathbb{Z}\Gamma}C_k(\widetilde{M}, \mathbb{Z}).\]

A \emph{standard Borel space} is a measurable space that is isomorphic to a Polish space with its Borel $\sigma$-algebra $\mathcal{B}$. Recall that a \emph{Polish space} is a separable completely metrizable topological space. Let $Z$ be a \emph{standard Borel probability space}, that is, a standard Borel space endowed with a probability measure $\mu$. Suppose now that $\Gamma$ acts (on the left) on such a standard Borel probability space $(Z, \mathcal{B}, \mu)$  
in a measurable and measure-preserving way. Denote this action by $\alpha\colon \Gamma\rightarrow \mathrm{Aut}(Z, \mu)$. Set
\[
L^\infty(Z, \mu;\mathbb{Z})=\left\{f\colon Z\longrightarrow \mathbb{Z}\,\vline\, \int_Z |f|d\mu<\infty\,\right\}.
\]
We define a right $\Gamma$-action on $L^\infty(Z, \mu;\mathbb{Z})$ by setting 
\[
(f\cdot \gamma)(z)=f(\gamma z), \,\forall f\in L^\infty(Z, \mu;\mathbb{Z}), \,\forall \gamma\in\Gamma, \, \forall z\in Z.
\]
There is a natural inclusion  for any $k\in \mathbb{N}$:
\[
\begin{array}{lrll}
i_\alpha\colon &C_k(M, \mathbb{Z})\cong \mathbb{Z}\otimes_{\mathbb{Z}\Gamma}C_k(\widetilde{M}, \mathbb{Z})&\longrightarrow &L^\infty(Z, \mu;\mathbb{Z})\otimes_{\mathbb{Z}\Gamma}C_k(\widetilde{M}, \mathbb{Z})\\
&1\otimes \sigma& \longmapsto &\mathrm{const}_1\otimes \sigma,
\end{array} 
\]
where the tensor products are taken over the given actions. We will write $C_\ast(M; \alpha)$ for the complex $L^\infty(Z, \mu;\mathbb{Z})\otimes_{\mathbb{Z}\Gamma}C_\ast(\widetilde{M}, \mathbb{Z})$. 
We call its elements \emph{parametrized chains}.

Given a parametrized $k$-chain $z=\sum_if_i\otimes\sigma_i\in L^\infty(Z, \mu;\mathbb{Z})\otimes_{\mathbb{Z}\Gamma}C_k(\widetilde{M}, \mathbb{Z})$ we define its \emph{parametrized $\ell_1$-norm} as \[|z|_1=\sum_i \int_Z |f_i|d\mu ,\]
where $\sigma_i\colon \Delta^k\rightarrow M$ is a singular simplex of dimension $k$. Here we assume that $z$ is in reduced form, that is, all the singular simplices $\sigma_i$ belong to different $\Gamma$-orbits. For more details on the integral foliated simplicial volume, see for example \cite{LoehPagliantini2016}.

Note that the action of $\Gamma$ on $\widetilde{M}$ restricts to the preimage of the boundary of $M$ under the covering map $p\colon\widetilde M\rightarrow M$. Thus we get an action 
\[\Gamma\longrightarrow \mathrm{Homeo}(p^{-1}(\partial M)).\]
This restricted action allows us to define the subcomplex
\[C_*(\partial M; \alpha)=L^\infty(Z, \mu;\mathbb{Z})\otimes_{\mathbb{Z}\Gamma}C_*(p^{-1}({\partial M}), \mathbb{Z})\subset C_*(M; \alpha),\]
and hence the quotient
\[C_*(M, \partial M; \alpha)=C_*(M; \alpha)/C_*(\partial M; \alpha).\]
This last quotient is naturally isomorphic to the chain complex 
\[
L^\infty(Z, \mu;\mathbb{Z})\otimes_{\mathbb{Z}\Gamma}C_\ast(\widetilde{M}, p^{-1}(\partial M); \mathbb{Z}).
\]
From now on we set:
\[
H_*(M, \partial M; \alpha)=H_*(C_*(M, \partial M; \alpha)).
\]
The \emph{(relative) integral foliated simplicial volume} of $M$ is the infimum over the parametrized $\ell_1$-norms of the (relative) parametrized cycles representing the fundamental class:
\[
\oldnorm{M, \partial M} =\inf_{\alpha, (Z, \mu)}\left\{\oldnorm{ M, \partial M}^\alpha\,\vline\, \alpha\colon \Gamma\longrightarrow \mathrm{Aut}(Z, \mu)\right\}, 
\]
where $\oldnorm{ M, \partial M}^\alpha$ is given by
\[
\inf\left\{\sum_i  \int_Z |f_i|d\mu\,\vline\, [M, \partial M]^\alpha=\left[\sum_i f_i\otimes\sigma_i\right]\in H_n(M, \partial M; \alpha)\right\}.
\]
Here $[M, \partial M]^\alpha$ denotes the image of the fundamental class $[M, \partial M]\in H_n(M, \partial M;\mathbb{Z})$ in $H_n(M, \partial M; \alpha)$ under the map induced by $i_\alpha$.
\begin{remark}
For every compact connected oriented $n$-manifold $M$, we have the following inequality (see \cite[Proposition 4.6]{LoehPagliantini2016}):
\[
\|M, \partial M\|\leq \oldnorm{M, \partial M}.
\]
\end{remark}
From now on, we consider only smooth connected orientable manifolds.
 
\subsection{Foliations}

We proceed to define a smooth regular foliation on  a\linebreak smooth connected manifold  $M$ and state some of its properties. 
A smooth \emph{regular foliation}  of $M$ is a partition $\fol = \{L_p\mid p\in M\}$ of $M$ where each \emph{leaf} $L_p$ is an embedded smooth submanifold $L_p\subset M$, called the \emph{leaf through $p$}, and such that the tangent spaces of the leaves are smooth  subbundles of $TM$ (see \cite[p.~9 (iii)]{Moerdijk}).
\begin{remark}
Observe that there might exist foliations where the leaves are not embedded submanifolds, cf.\ to \cite[Section~1.1]{Moerdijk}, but we will not consider such foliations in the present work. 
\end{remark}
The following phenomena give rise to foliations: Let $p\colon M\to B$ be a smooth submersion. Then the partition induced by the fibers of $p$ produces a foliation. More generally, any involutive subbundle of $TM$ induces a foliation on $M$. Given a compact Lie group $G$ acting smoothly on $M$, we define the \emph{isotropy subgroup} at $p\in M$ as $G_p = \{g\in G\mid g\cdot p = p\}$. 
We say that a compact Lie group $G$ acts \emph{almost freely} on $M$ if $G_p$ is finite for any point $p\in M$. The set $G(p)= \{g\cdot p \mid g\in G\}$ is called the \emph{orbit through $p$}. The partition given by the orbits defines a regular foliation $(M,\fol)$ if the dimension of the isotropy groups $G_p$ is constant with respect to $p$ 
(see \cite[Theorem~3.65 and Example~5.3]{Alexandrino}\cite[p.~16]{Moerdijk}). We will refer to such foliations as \emph{homogeneous foliations}.
The \emph{codimension} of a foliation on a connected manifold is the codimension of any leaf.  
The quotient space $M/\fol$ induced by the partition, equipped with the quotient topology, is called the \emph{leaf space of $\fol$}. We denote the projection map by $\pi\colon M\to M/\fol$, and the image under $\pi$ of a subset $A\subset M$ by $A^*$.


\subsubsection{Holonomy}

Given a leaf $L\in\fol$ and a point $x\in L$, we can select a transversal section $T$  at $x$, \emph{i.e.} a linear subspace of $T_x M$ such that $T_x M = T_xL\oplus T$. There is a well defined action of $\pi_1(L,x)$ on $T$ as follows: given a closed loop $\alpha\colon [0,1]\to L$, starting at $x$, and a vector $v\in T$, we may use foliated charts to extend $v$ to a vector field $V$ on $\alpha$, such that if $\gamma_t\colon [0,1]\to M$ is the integral curve of $V(\alpha(t))$, then $\gamma_t(1)$ is contained in the same leaf of $\fol$ for all $t\in[0,1]$. 
This action only depends on the homotopy class of the loop $\alpha$. Furthermore this action is independent of the chosen transversal $T$ (see \cite[2.1]{Moerdijk}). Let $q$ denote the codimension of the foliation; then by identifying $T$ with $\R^q$ we get a group morphism:
\[
	\hol(L,x)\colon \pi_1(L,x)\longrightarrow \Diff_0(\R^q).
\]
Here $\Diff_0(\R^q)$ denotes the group of germs of diffeomorphisms of $\R^q$ fixing the origin. Observe that the identification of $T$ with $\R^q$ is via foliated charts. Thus the map $\Hol(L,x)$ is determined up to conjugation in $\Diff_0(\R^q)$ (see \cite[p.~23]{Moerdijk}). We define the \emph{holonomy group of the leaf $L$}, denoted by  $\Hol(L,x)$, to be  the image of $\pi_1(L,x)$ under the morphism $\hol(L, x)$. This group is independent of the base point $x$ up to conjugation in $\Diff_0(\R^q)$. We denote by $K$ the kernel of the short exact sequence:
\begin{equation}\label{EQ: holonomy short exact sequence}
1\longrightarrow K\longrightarrow \pi_1(L,x)\longrightarrow \Hol(L,x)\longrightarrow 1.
\end{equation}
The holonomy group of a leaf measures roughly ``how twisted" the foliation is around the leaf. A foliation induced by a submersion has trivial holonomy for any leaf. Observe that, by considering germs of the derivative at the origin, we obtain a new representation:
\[
	\mathrm{Dhol}(L,x)\colon\pi_1(L,x) \longrightarrow \mathrm{GL}_q(\R).
\]
Given a homogeneous foliation by a compact Lie group $G$, via an auxiliary equivariant metric we note that the isotropy group $G_p$ at $p$ acts on a transversal section $T$ at $p$. The holonomy is the quotient of the connected component $(G_p)_0$ containing the identity by the kernel of the action (see \cite[Section~3.1]{Radeschi15}). As in the case of group actions, a \emph{principal leaf} is a leaf  with trivial holonomy.

We now describe a tubular neighborhood of a leaf $L$ in $M$ with finite holonomy. Let $\bar{L}\to L$ be the covering space of $L$ associated with the subgroup $K$ in \eqref{EQ: holonomy short exact sequence}. Consider a transversal section $T$ at a point $p\in L$ and the holonomy action of the fundamental group $\pi_1(L, p)$ on $T$ described above. Then we have the following theorem:

\begin{thm}[Local Reeb stability, Theorem~2.9 in \cite{Moerdijk}]\label{T: Tubular Neighborhood of compact leaf}
Let $(M,\fol)$ be a regular foliation. For a compact leaf $L$ with finite holonomy $H= \Hol(L,p)$, there exists a foliated open neighborhood $V$ of $L$ in $M$ and a diffeomorphism
\[
	\bar{L}\times_{H} T \longrightarrow V.
\]
\end{thm}
In the case when $\fol$ is induced by a group action, Theorem~\ref{T: Tubular Neighborhood of compact leaf} is known as the Slice Theorem (see \cite[Theorem~3.57]{Alexandrino}). We say that a regular smooth foliation $\fol$ has \emph{finite holonomy} if every leaf has finite holonomy. For more information on foliations, see for example \cite{Moerdijk, CandelConlon2000}.


\subsubsection{Orbifolds}
We recall the notion of an orbifold for the sake of completeness. Consider a topological space $X$, and fix $n\geqslant 0$. An \emph{$n$-dimensional orbifold chart} on $X$ is given by an open subset $\widetilde{U}\subset \R^n$, a finite group $G$ of smooth automorphisms of $\widetilde{U}$, and a $G$-invariant map $\phi\colon \widetilde{U}\to X$, which induces a homeomorphism of $\widetilde{U}/G$ onto some open subset $U\subset X$. 
By $G$-invariant we mean that $\phi(g\cdot u) =\phi(u)$ for any $g\in G$ and $u\in \widetilde{U}$. An \emph{embedding} $\lambda\colon (\widetilde{U},G,\phi)\hookrightarrow (\widetilde{V},H,\psi)$ between two orbifold charts is a  smooth embedding $\lambda\colon \widetilde{U}\to \widetilde{V}$ such that $\psi\lambda = \phi$. 
Given two charts $(\widetilde{U},G,\phi)$ and $(\widetilde{V},H,\psi)$ with $\phi(\widetilde{U}) = U$ and $\psi(\widetilde{V}) = V$,  we say that they are compatible if for any $x\in U\cap V$ there exists $W \subset U\cap V$, an open neighborhood of $x$, and a chart $(\widetilde{W},K,\mu)$ with $\mu(\widetilde{W}) =W$, such that there are embeddings $(\widetilde{W},K,\mu)\hookrightarrow (\widetilde{U},G,\phi)$ and $(\widetilde{W},K,\mu)\hookrightarrow (\widetilde{V},H,\psi)$. 
The space $X$ equipped with an atlas of orbifold charts is called an \emph{orbifold}, and we use the notation $\mathcal{O}_X$ to distinguish it from the original topological space $X$. 

Let $(M, \fol)$ be a regular foliation with all leaves compact with finite holonomy. From Theorem~\ref{T: Tubular Neighborhood of compact leaf}, an open neighborhood of $p^*\in M/\fol$ is given by $T/\Hol(L_p,p)$. Thus, the transversal sections to the leaves, the holonomy, and the projection map $\pi$ induce an orbifold atlas on $M/\fol$.

\begin{thm}[Theorem~2.15 in \cite{Moerdijk}]\label{T: Orbifold structure of the leaf space}
Let $(M,\fol)$ be a foliation of codimension $q$ such that any leaf of $\fol$ is compact with finite holonomy group. Then the space of leaves $M/\fol$ has a canonical orbifold structure of dimension $q$.
\end{thm}

Consider an $n$-dimensional orbifold $X$, and for $x\in X$, let  $(\bar{U},G,\phi)$ be a chart around $x$. Take $y\in \bar{U}$ such that $\phi	(y) =x$. The \emph{local group} or \emph{isotropy} at $x$ is the conjugacy class (see \cite[Section~2.4]{Moerdijk}) in $\Diff_0(\R^n)$ of
\[
	G_x :=  \{g\in G\mid gy=y\}.
\]

Consider $(M,\fol)$ a regular foliation of codimension $q$, with compact leaves of finite holonomy. Given $p^\ast\in M/\fol$, the isotropy $G_p$ is the conjugacy class of $\Hol(L_p,p)$ in $\Diff_0(\R^q)$ (see \cite[Theorem~2.15]{Moerdijk}).

 
\subsection{Triangulation of the leaf space}\label{subs: Triangulation of leaf space}

Let $M$ be a compact manifold. We consider a regular foliation $(M,\fol)$ by circles, \emph{i.e.} any leaf $L$ of $\fol$ is homeomorphic to a circle. We observe that for this case the holonomy of any leaf $L$ is isomorphic to either the trivial group, $\Z/k\Z$,  for $k\in\mathbb{N}_{\geq 2}$, or $\Z$, since $\Z$, $k\Z$ and the trivial group are the only subgroups of $\Z = \pi_1(\Sp^1)$. We will assume from now on that $\Hol(L,x)$ is finite, \emph{i.e.} of the form $\Z/k\Z$ or the trivial group, for any leaf. In this case, by Theorem~\ref{T: Orbifold structure of the leaf space}, the leaf space $M/\fol$ is a compact orbifold. For a general regular foliation $(M,\fol)$ with compact leaves and finite holonomy we will now define a decomposition into strata, and recall that we can triangulate the leaf space $M/\fol$ with respect to this decomposition. The triangulation of the leaf space is such that the interior of each simplex lies in a smooth submanifold contained in $M/\fol$. Furthermore  the projection map $\pi\colon M\to M/\fol$ is a smooth submersion between smooth manifolds over the interior of each simplex.


\subsubsection{The stratification}\label{SSS: stratification}

We recall that a \emph{stratification} of a topological space $X$ is a partition of $X$ into subsets $\{\Sigma_\alpha\}_{\alpha\in \Lambda}$ such that:
\begin{enumerate}[(i)]
\item The partition is locally finite, \emph{i.e.} each compact subset of $X$ only intersects a
finite number of strata.
\item If $\Sigma_\beta\cap\Cl(\Sigma_\alpha)\neq \emptyset$, then $\Sigma_\beta\subset\Cl(\Sigma_\alpha)$.
\item If $X$ is a smooth manifold, the strata $\Sigma_\alpha$ are embedded smooth submanifolds.
\end{enumerate}

Let $H <\Diff_0(\R^n)$ be a finite subgroup. We denote by $(H)$ its conjugacy class in $\Diff_0(\R^n)$. For a regular foliation $(M,\fol)$ with compact leaves and finite holonomy,  we consider the set 
\[
\Sigma_{(H)} =\{p^*\in M/\fol \mid (\Hol(L_p,p))=(H)\}\subset M/\fol.
\]
From the description of the tubular neighborhood of a leaf $L_p$ given by Theorem~\ref{T: Tubular Neighborhood of compact leaf}, we can see that $\{\Sigma_{(H)}\mid H=\Hol(L_p,p)\}$ gives a stratification of $M/\fol$. 
We can describe $\Sigma_{(H)}$ locally as follows: for $p^\ast\in \Sigma_{(H)}$ we consider $\bar{L}_p\times_H T$, the tubular neighborhood of $L_p$ described in Theorem~\ref{T: Tubular Neighborhood of compact leaf}. The points in $\Sigma_{(H)}$ correspond to the projection of the fixed points of the action of $\pi_1(L)$ on $T$. Thus locally, the stratification of $M/\fol$  is  induced from the stratification of the orbit space $T/H$ (see \cite[Section~4.3]{Pflaum}, \cite[Sections~4.2 and 4.3]{Sniatycki}). From the fact that the transverse spaces $T$ to the leaves give an atlas for the orbifold $M/\fol$, we see that indeed the subsets  $\Sigma_{(H)}$ of $M/\fol$ induce a stratification on the  orbifold $M/
\fol$ (see \cite[Section~1.2]{MoerdijkPronk1999}).

\begin{remark}
In general for an orbifold $X$, the isotropy groups yield a stratification of $X$, by setting $\Sigma_{(H)} = \{x\in X\mid (H) = (G_x)\}$.
\end{remark}

For  a foliation by circles $(M,\fol)$ with finite holonomy over  a compact  manifold $M$ we describe the holonomy stratification. The stratification of $M/\fol$ is given by $\{\Sigma_{(\Z/k\Z)}\mid k\in \N\}$. Since $M$ is compact, there are finitely many conjugacy types of holonomy groups. 

\begin{prop}\label{P: dimension of holonomy stratum}
Let $(M,\fol)$ be  an $n$-dimensional manifold with a non-trivial regular foliation of codimension $k$ at most $n-1$.  Assume that  each leaf is closed, and has finite holonomy. For a non-trivial holonomy group, its holonomy stratum has dimension at most $k-1$.
\end{prop}

\begin{proof}
Fix $p\in M$ such that $L_p$ has non-trivial holonomy. Then by taking the germs at $0$ of elements in $\Hol(L_p,p)\subset \Diff_0(\R^{k})$, we have a faithful representation of $H=\Hol(L_p,p)$ into $\GL_{k}(\R)$. Let $\bar{L}\times_{\pi_1(L_p,p)} T$ be the tubular neighborhood of $L_p$ given by Theorem~\ref{T: Tubular Neighborhood of compact leaf}. Since we have a faithful representation of $\pi_1(L_p,p)$ into $\GL_{k}(\R)$, this neighborhood is determined by the linear action of $H$ on $\R^{k}$. 
We note that the set of fixed points $\mathrm{Fix}(H,\R^{k})$ is a non trivial linear subspace of dimension $q$.  Thus we have $0< q < k$. 
Moreover, the set of fixed points corresponds to leaves in the same holonomy stratum as $p^\ast$.   
Thus the connected component of the  holonomy stratum  containing $p^\ast$  has dimension $q \leqslant  k-1$.
\end{proof}

\begin{remark}\label{R: non-trivial holonomy stratum circle foliation has dimension at most n-2}
For a compact manifold $(M,\fol)$ with a regular foliation by circles, the previous proposition shows that  for any non-trivial holonomy, a connected component of its stratum has dimension at most $n-2$.
\end{remark}

\subsubsection{Triangulation}\label{SSS: triangulation}

Given an orbifold $X$, there exists a triangulation $\Tr$ of $X$, such that the closures of the strata $\Sigma_{(H)}$ are contained in subcomplexes of $\Tr$ (see \cite{Goresky1978,Yang1963}). By taking a subdivison of $\Tr$ we can assume that for a simplex $\sigma$ of $\Tr$,  the isotropy groups  of the points in the interior of $\sigma$  are the same, and are subgroups of the isotropy groups of the points in the boundary of $\sigma$. Furthermore we can assume that there is a face $\sigma'\subset \sigma$, such that the isotropy is constant on $\sigma\setminus \sigma'$, and possibly larger on $\sigma'$. Thus, any simplex $\sigma$ of $\Tr$ has a vertex $v\in \sigma$ with maximal isotropy. This means that for any point $x\in \sigma$, a conjugate of $G_x$ is contained in $G_v$.


	
\subsection{Hollowings}\label{S: Hollowings}

In this section, we introduce our main geometric tool, \emph{hollowings}. The right category to consider this construction is that of \emph{manifolds with corners}, so we first define these for the sake of completeness. We follow the presentations in \cite[Appendix B]{Fauserthesis} and \cite{Yano1982}, and refer the reader interested in more details on manifolds with corners and their submanifolds to \cite[Appendix B]{Fauserthesis}.

Let $n\in \mathbb{N}$ and let $k\in \{0,..., n\}$. A \emph{sector of dimension $n$
with index $k$} is a set of the form
\[
A =\{(x_1,...,x_n) \in \mathbb{R}^n \mid x_{i_1} \geq 0,..., x_{i_k}\geq 0\}
\]	
with pairwise distinct $i_1,..., i_k\in\{1,..., n\}$. Note that $\mathbb{R}^n$ is a sector
of dimension $n$ with index $0$. The \emph{index} of a point $x\in A$ is the number of its coordinates that are equal to $0$ among $x_{i_1}, ..., x_{i_k}$. The index of $0\in \mathbb{R}^n$ is $k$, the index of $A$.

Let $m\in \mathbb{N}$ be arbitrary and $A\subset  \mathbb{R}^n$ a sector. Let $U\subset A$ be open, and $V\subset \mathbb{R}^m$ an arbitrary subset. A map $\phi\colon U\rightarrow V$ is \emph{smooth} if for every $x\in U$, there exists an open neighborhood $U_x\subset \mathbb{R}^n$ of $x$ and a smooth  map $\phi_x\colon U_x\rightarrow V$ such that $\phi_x$ and $\phi$ coincide on $U\cap U_x$.

Let $M$ be a topological space.
An \emph{atlas with corners for $M$ of dimension $n$} is a family $(U_i,V_i, \phi_i)_{i\in I}$ where $(V_i)_{i\in I}$
is an open cover of $M$, each $U_i$ is an open subset in a sector $A_i$ of dimension $n$
and each $\phi_i\colon U_i\rightarrow V_i$ is a homeomorphism (called \emph{chart}) such that all transition maps
are smooth, \emph{i.e.} for all $i, j \in I$, the map
\[
(\phi_j^{-1}\circ\phi_i)|_{U_i\cap U_j}
\]
is smooth (in the above sense) and is a smooth diffeomorphism
onto its image.


We say that a Hausdorff second countable topological space $M$ is an \emph{$n$-dimensional manifold with corners}, if it admits an atlas with corners as above. For such an $M$, define the subsets $\partial M= <^{(1)}M\supset... \supset <^{(n)}M$ as the image under the chart homeomorphisms of all subsets of points of index at least $1\leq ...\leq n$. This does not depend on the choice of the charts.

In order to define submanifolds with corners, we need the notion of \emph{subsector}. Let $A=\{x\in \mathbb{R}^n \mid x_{i_1} \geq 0,..., x_{i_k}\geq 0\}$ be a sector of dimension $n$ with index $k$. 
Let $p, \ell, j \in \mathbb{N}$ with $p\geq \ell$. A \emph{subsector of $A$ (of codimension $p$, coindex
$\ell$ in $A$ and complementary index $j$)} is a subset $A'\subset A$ determined by 
sets $L\subset\{i_1,... , i_k\}$ with $\ell$ elements, $P\subset\{1,..., n\}\setminus\{i_1,... , i_k\}$ with $p-\ell$ elements and $J\subset \{1,..., n\}\setminus\left(\{i_1,... , i_k\}\cup P\right)$
with $j$ elements such that
\[
A' = \{(x_1, ...,x_n)\in A \mid x_i = 0\,\forall i\in L, x_i = 0\,\forall i\in P\,, x_i\geq 0\,\forall i\in J\}.
\]

Let $M$ be a manifold with corners. A
closed subspace $N\subset M$ is called \emph{submanifold (with corners)}, if for all $x\in N$ there
exists a chart $\phi\colon U \rightarrow V$ of $M$, where $U$ is contained in a sector $A$, and a subsector $A'\subset A$ such that
$\phi(0) = x$ and $\phi^{-1}(V \cap N) = U \cap A'$.

With these definitions, we are able to define \emph{hollowings}. Let $M$ be a manifold with corners and $N$ a submanifold of $M$ with corners, such that $N$ is transverse to each 
$(<^{(k)}M)\setminus (<^{(k+1)}M)$ and $<^{(k)}N=N\cap<^{(k)}~M$. The tubular neighborhood $\nu(N)$ of $N$ in $M$ has the structure of a disk bundle over $N$. Let $\psi\colon \nu_S(N)\times [0, 1]\rightarrow \nu(N)$ be the parametrization of $\nu(N)$ in polar coordinates, that is $\nu_S(N)$ is the total space of the associated sphere bundle, $\psi|_{\nu_S(N)\times\{1\}}=\mathrm{id}_{\nu_S(N)}$ and $\psi|_{\nu_S(N)\times\{0\}}$ is the projection of the bundle $\nu(N)\rightarrow N$. Take 
\[
M'=\mathrm{cl}\left(M\setminus \nu(N)\right)\bigcup_{\psi|_{\nu_S(N)\times\{1\}}} \nu_S(N)\times[0, 1].
\] 
There is a natural map $p\colon M'\rightarrow M$ defined by $p|_{M\setminus \nu(N)}=\mathrm{id}_{M\setminus \nu(N)}$ and $p|_{\nu_S(N)\times[0, 1]}=\psi$. The space $M'$ has a canonical structure as a manifold with corners, making $p$ a differentiable map between manifolds with corners.

The \emph{hollowing} of $M$ at $N$ is the  map $p\colon M'\rightarrow M$. For the case when $N=\emptyset$, we define $M'=M$ and $p=\mathrm{id}_M$. The submanifolds $N\subset M$ and $p^{-1}(N)\subset M'$ are called the \emph{trace} and the \emph{hollow wall} of $p$, respectively. For a submanifold $L\subset M$, denote by $\bar{p}(L)$ the submanifold $\mathrm{cl}(p^{-1}(L\setminus(L\cap N)))\subset M'$.

Let $(M,\fol)$ be a compact oriented $n$-dimensional manifold, with possibly empty boundary $\partial M$, equipped with a regular foliation  of dimension $q$, such that each leaf has finite holonomy. Consider a triangulation of the leaf space $M/\fol$ as described in Section \ref{SSS: triangulation}, and denote by $(M/\fol)^{(k)}$ the $k$-th skeleton of the triangulation. We define a sequence of hollowings
\[
\xymatrix
{
M_{n-q-1}\ar[r]_{p_{n-q-2}}&M_{n-q-2}\ar[r]&\cdots\ar[r]&M_1\ar[r]_-{p_0}&M_0=M
}
\]
in the following way.  For $k=0, ..., n-q-2$, we define inductively the map $p_k\colon M_{k+1}\rightarrow M_k$ to be the hollowing at 
\[
X_k = \bar{p}_{k-1}(\ldots \bar{p}_0(\pi^{-1}((M/\fol)^{(k)}))\ldots),
\]
for $k\geqslant 1$, and $X_0 = \pi^{-1}((M/\fol)^{(0)})$. The following proposition shows that at each step of the hollowing we obtain a foliated manifold.

\begin{prop}\label{P: The hollowing admits a foliation}
Let $(M,\fol)$ be a compact  smooth manifold with a regular foliation with finite holonomy. For each hollowing $p_i\colon M_{i+1} \to M_{i}$, there are  regular foliations  $\fol_i$ on $M_i$, and $\fol_{i+1}$ on $M_{i+1}$, such that $p_i$ maps $\fol_{i+1}$ to $\fol_i$.
\end{prop}

\begin{proof}

Consider $U$ a tubular neighborhood of $X_i$. Fix $q\in U$. Then $q$ belongs to a tubular neighborhood of some leaf $L_p$ with $p\in X_i$.  Denote by $H$ the holonomy group of $L_p$, and $T$ a transverse subspace to $T_p L_p$. Since a tubular neighborhood of $L_p$ is foliated diffeomorphic to $\bar{L}_p\times_{\Hol(L_p,p)} T$, the leaf $L_q$ is contained in this tubular neighborhood, and thus in the tubular neighborhood $U$ of $X_i$. This implies that  the tubular neighborhood of $X_i$ in $M_i$  is foliated. 
\end{proof}

\begin{remark}
In general, let $(M,\fol)$ be a compact  smooth manifold with a regular foliation with finite holonomy. 
Consider a fixed stratum $\Sigma_{(H)}$ of $M/\fol$. For $p^\ast\in \Sigma_{(H)}$ we have, from Theorem~\ref{T: Tubular Neighborhood of compact leaf}, that a neighborhood is given by $T/H$. Furthermore, we may assume that $T$ is homeomorphic to an $(n-q)$-disk, where $q$ is the dimension of $\fol$, and $H$ acts linearly on it. This implies that the fixed point set of the action of $H$ on $T$ is a linear subspace. Thus locally, the stratum is a submanifold in $T/H$. Since, by construction, the simplices of the triangulation are contained in strata,  from  the description of the hollowing, and the local description of the foliation, we see that the map $p^\ast_{i}\colon M_{i+1}/\fol_{i+1}\to M_{i}/\fol_{i}$ is a hollowing on the $i$-th skeleton.
\end{remark}

\begin{remark}
For the particular case that  the regular foliation $\fol$ is given by circles, then the foliation $\fol_i$ is also given by circles.
\end{remark}



\subsubsection{Examples}\label{SS: examples of hollowings}

We present some examples of hollowings.

\begin{example}\label{E: examples of hollowings}
\hfill
\begin{enumerate}
\item\label{SE: Example of hollowing of simplex} For the standard $n$-dimensional simplex $\Delta^n$, we denote by $(\Delta^n)^{(k)}$ the $k$-th skeleton of $\Delta^n$. We can define a sequence of hollowings inductively. Setting $\Delta^n_0 = \Delta^n$, we define $q_k\colon \Delta^n_{k+1}\to \Delta^n_k$ the hollowing at $\bar{q}_{k-1}\cdots \bar{q}_0((\Delta^n)^{(k)})$ (see Figure \ref{F: hollowed simplex}).
\begin{figure}[h]
\centering
\begin{subfigure}[c]{0.3\textwidth}
\centering
\includegraphics[scale=0.5]{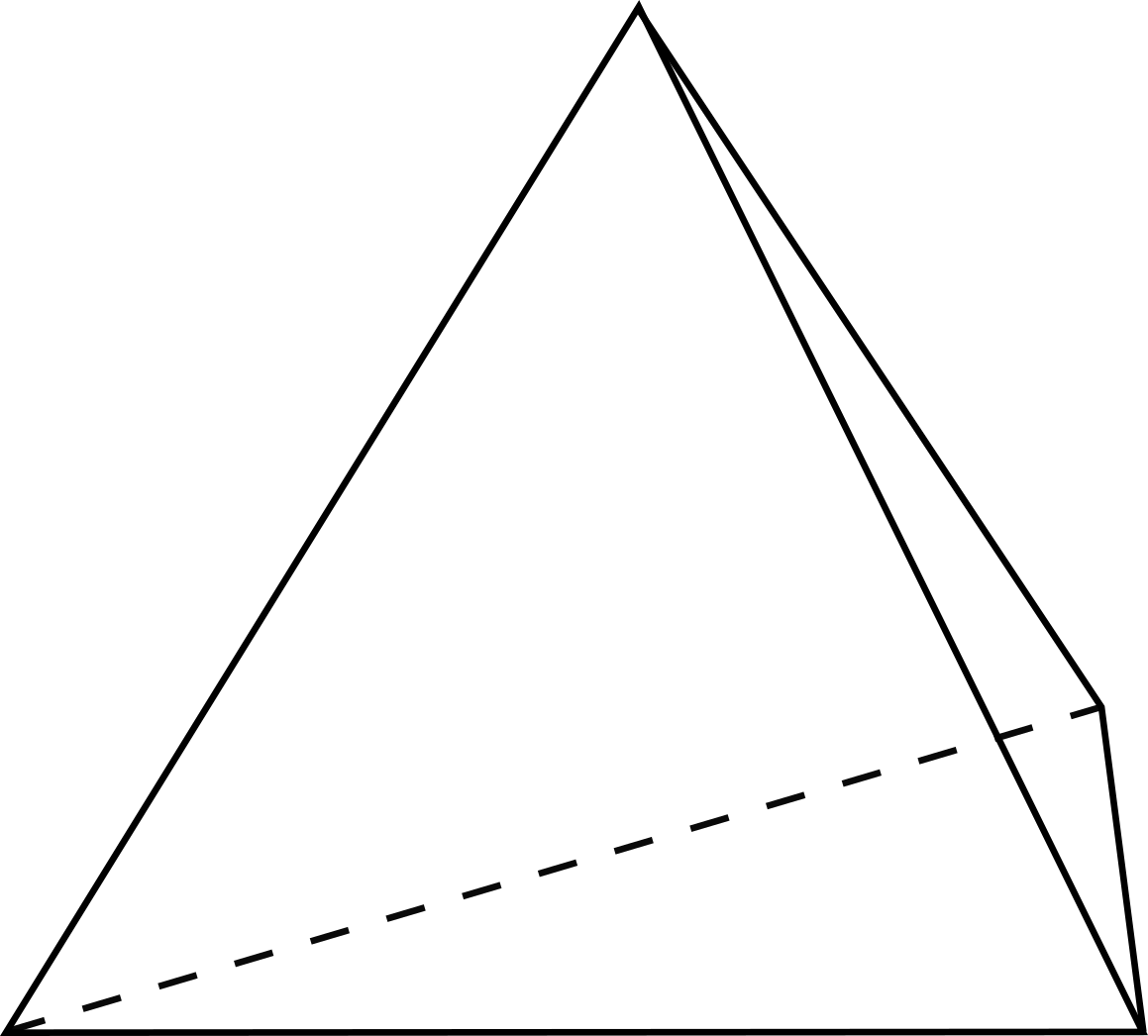}
\caption{The simplex $\Delta^3_0$}
\end{subfigure}
\begin{subfigure}[c]{0.3\textwidth}
\centering
\includegraphics[scale=0.5]{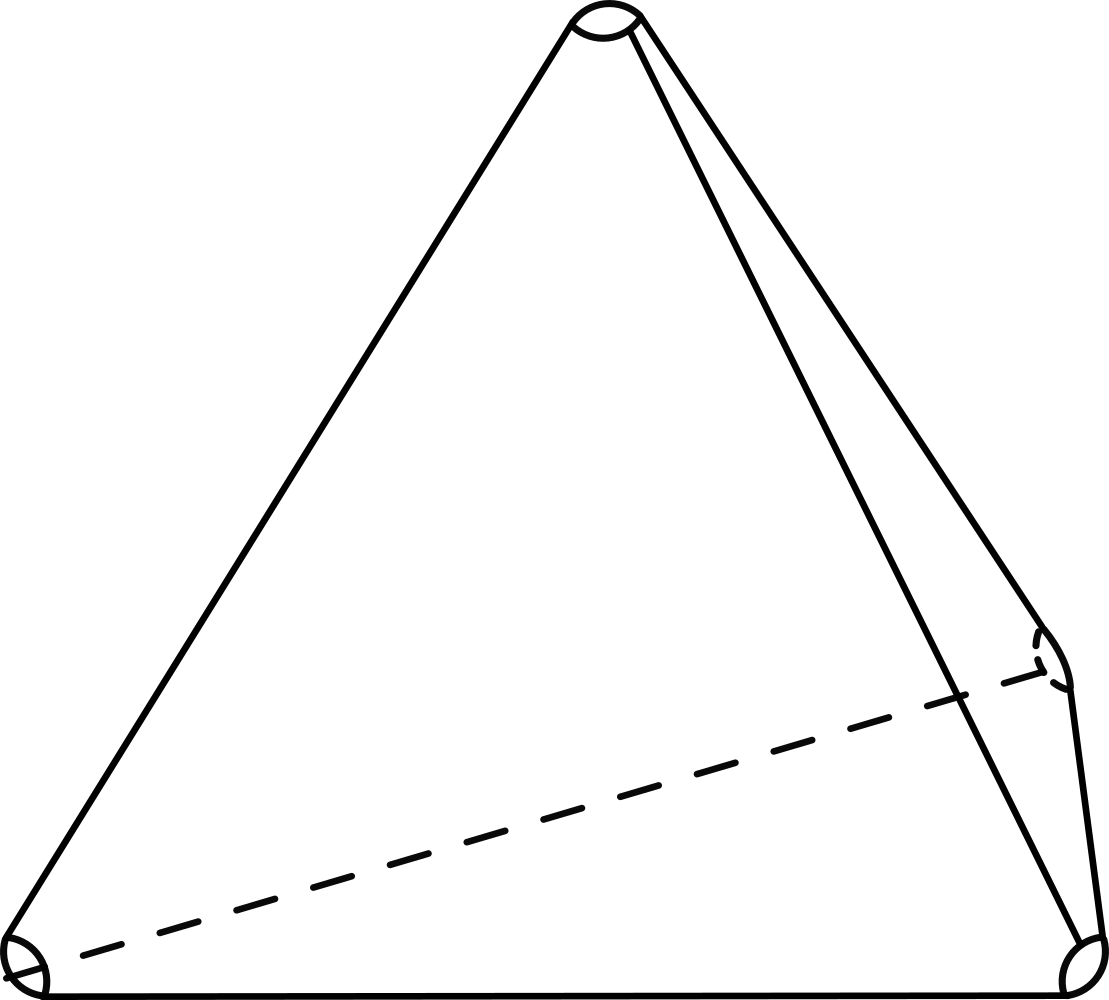}
\caption{The hollowed simplex $\Delta^3_1$}
\end{subfigure}
\begin{subfigure}[c]{0.3\textwidth}
\centering
\includegraphics[scale=0.5]{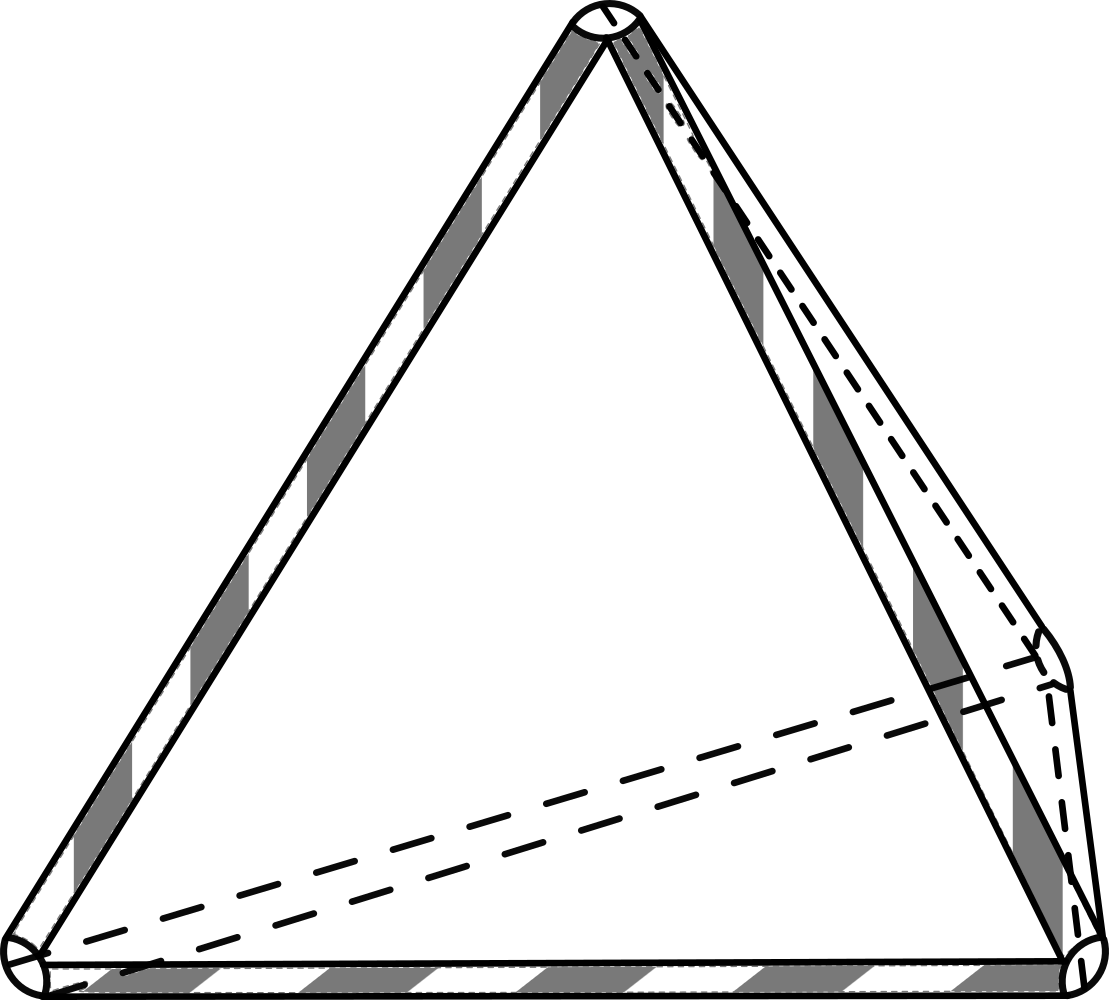}
\caption{ The hollowed simplex $\Delta^3_2$}
\end{subfigure}
\caption{Hollowings of the simplex}
\label{F: hollowed simplex}
\end{figure}
\item Let $\Sp^1$ act on $\Sp^2$ by rotations along the north-south axis. This gives a foliation $\fol$ of $\Sp^2$ by the orbit circles, except for the two fixed points $N, S$. The leaf space $(\Sp^2)^\ast$ is a segment with orbifold structure given by the open segment and its two boundary points corresponding to the fixed points $N, S$.

A hollowing $p$ at the submanifold $\{N, S\}$ is described as follows: in $M'$ the tubular neighborhoods $\nu(N), \nu(S)\subset \Sp^2$ of $N$ and $S$ (two $2$-disks) are replaced by the disk bundles over $N$ and $S$ and glued to $\Sp^2\setminus \left(\nu(N)\cup\nu(S)\right)$ by the identity map along the associated sphere bundles (see Figure \ref{S^2}). Thus in this case the hollowed manifold is again $\Sp^2$.

\begin{figure}[h]
\centering
\begin{subfigure}[b]{0.6\textwidth}
\centering
      \begin{tikzpicture}[scale=0.5]     
  \draw (0,0) circle (2cm);
\draw (-2.5,0.5) node[above]{$\Sp^2$} ;
 \draw (5.9,0) node[above]{$(\Sp^2)^*$} ;
  \draw (-2,0) arc (180:360:2 and 0.6);
  \draw[dashed] (2,0) arc (0:180:2 and 0.6);
  \draw (-1.8,0.87) arc (180:362:1.77 and 0.6);
  \draw[dashed] (1.8,0.87) arc (3:180:1.75 and 0.6);
  \draw (-1.8,-0.87) arc (182:363:1.78 and 0.6);
  \draw[dashed] (1.8,-0.87) arc (0:180:1.75 and 0.6);
  \fill[fill=black] (0,2) circle (3pt) node[above]{$N$} (0,2);
	\fill[fill=black] (0,-2) circle (3pt) node[below]{$S$} (0,2);
	\draw (4.7, -2)--(4.7, 2);
	\fill[fill=black] (4.7,2) circle (3pt) node[above]{$N^*$} (0,2);
	\fill[fill=black] (4.7,-2) circle (3pt) node[below]{$S^*$} (0,2);
	\draw[->] (2.2,0) .. controls +(0.8,0.3) .. (4.5,0);
\end{tikzpicture}
\caption{Foliation and leaf space}
\end{subfigure}	

\begin{subfigure}[b]{0.39\textwidth}
\centering
\begin{tikzpicture}[scale=0.5]
\draw (7,0) circle (2cm);
\draw (9.9,0.5) node[above]{$(\Sp^2)'$} ;
  \draw (5,0) arc (180:360:2 and 0.6);
  \draw[dashed] (9,0) arc (0:180:2 and 0.6);
  \draw (5.2,0.87) arc (180:362:1.77 and 0.6);
  \draw[dashed] (8.8,0.87) arc (3:180:1.75 and 0.6);
  \draw (5.2,-0.87) arc (182:363:1.78 and 0.6);
  \draw[dashed] (8.8,-0.87) arc (0:180:1.75 and 0.6);
  \draw (6.7, 1.97) arc (200:340:0.3);
  \draw[dashed] (7.3,-1.97) arc (20:160:0.3);
  \fill[fill=black] (7,2) circle (1pt) ;
	\fill[fill=black] (7,-2) circle (1pt) ;
\end{tikzpicture}
\caption{Hollowing}
\end{subfigure}
      \caption{Action of $\Sp^1$ by rotations on $\Sp^2$.} \label{S^2}
     \end{figure}
\item Let $\Sp^3=\{(z_0, z_1)\in \mathbb{C}^2\,\mid\,|z_0|^2+|z_1|^2=1\}$. Fix $p, q\in \mathbb{Z}$ with $(p, q)=1$ and let $\Sp^1$ act on $\Sp^3$ via
\[e^{2i\pi\theta}\cdot(z_0, z_1)=(e^{2i\pi\theta p}z_0,e^{2i\pi\theta q}z_1). \]
The action has no fixed points, but the circles  
\[
\{(z_0, 0) \mid |z_0|=1\},\, \{(0, z_1) \mid |z_1|=1\}\subset \Sp^3,
\]
have non-trivial isotropy isomorphic to $\mathbb{Z}_p$ and $\mathbb{Z}_q$, respectively. All other points have trivial isotropy. The orbit of each point is a circle. The orbit space is $\Sp^2$, with orbifold structure a cylinder with upper and lower boundary circles collapsed to points $x_p^\ast$ and $x_q^\ast$, respectively, corresponding to the two orbits with isotropy $\Z_p$ and $\Z_q$, respectively. Take a triangulation of this $\Sp^2$ such that $x_p^\ast, x_q^\ast$ are among the vertices: for example add four vertices $y_0^\ast, ..., y_3^\ast$ on the equator and join each of them by edges to $x_p^\ast$ and $x_q^\ast$ (see Figure~\ref{F: Hopf fibration}). In this fashion, we obtain a triangulation of $\Sp^2$ with $6$ vertices, $12$ edges and $8$ triangles.  The  hollowings consist here of a single map $p_0\colon M_1\rightarrow \Sp^3$, that is the hollowing at $\pi^{-1}(\Sp^3/\fol^{(0)})$: we take tubular neighborhoods of the $6$ circles $\pi^{-1}(x_p^\ast), \pi^{-1}(x_q^\ast), \pi^{-1}(y_0^\ast), ..., \pi^{-1}(y_3^\ast)$ out of $\Sp^3$ and glue back disk bundles over these circles using the identity map along the associated sphere bundles.
\begin{figure}[h]
\centering
\begin{subfigure}[c]{0.3\textwidth}
\centering
\includegraphics[scale=0.5]{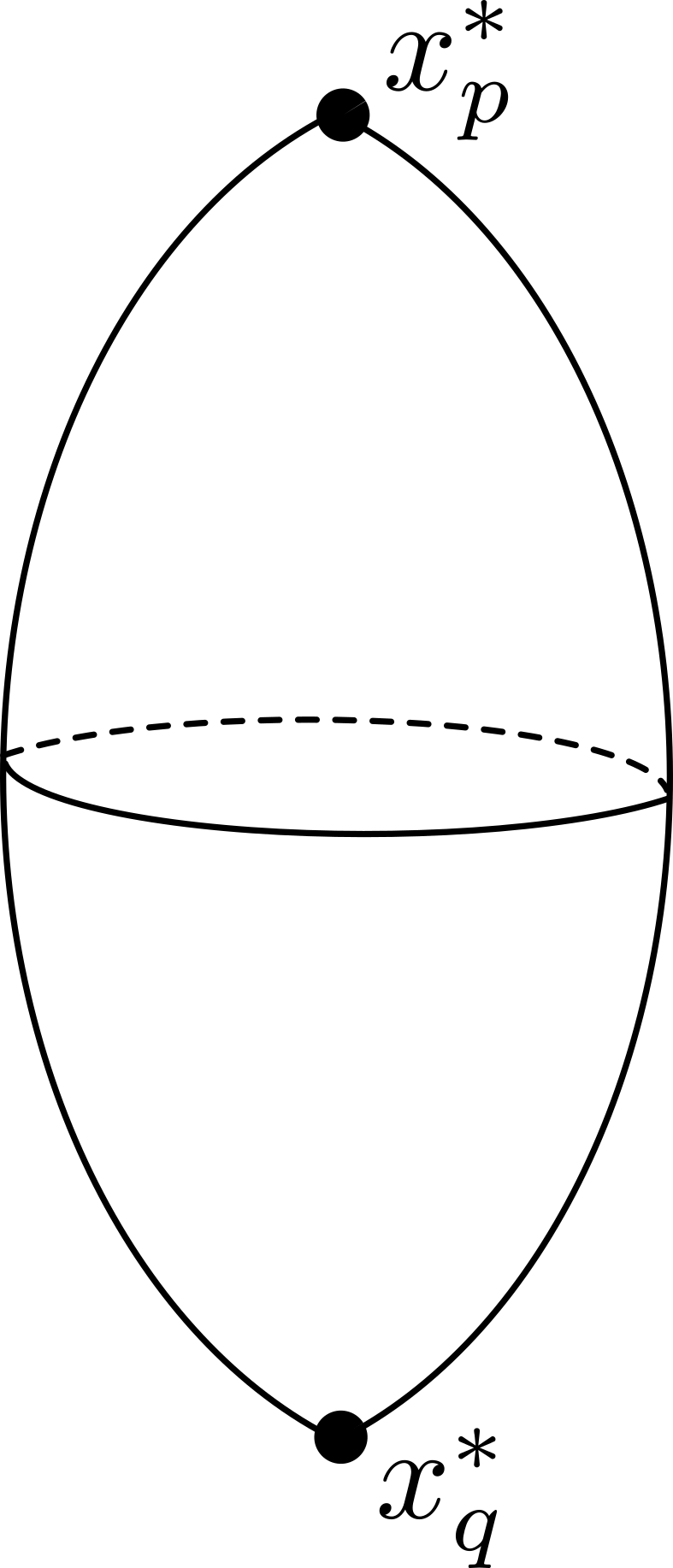}
\caption{Orbit space}
\end{subfigure}
\begin{subfigure}[c]{0.3\textwidth}
\centering
\includegraphics[scale=0.5]{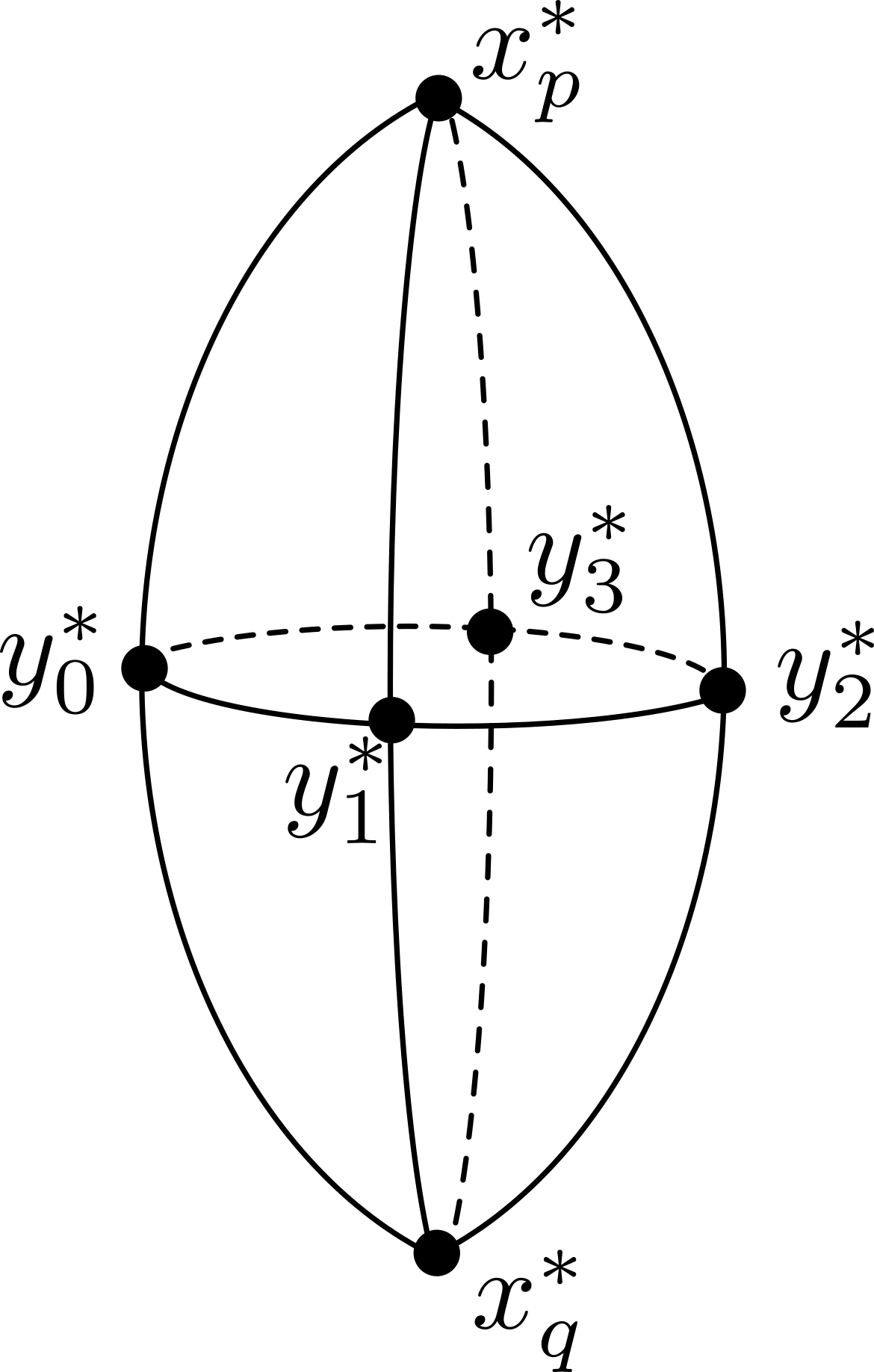}
\caption{Triangulated orbit space}
\end{subfigure}
\begin{subfigure}[c]{0.3\textwidth}
\centering
\includegraphics[scale=0.5]{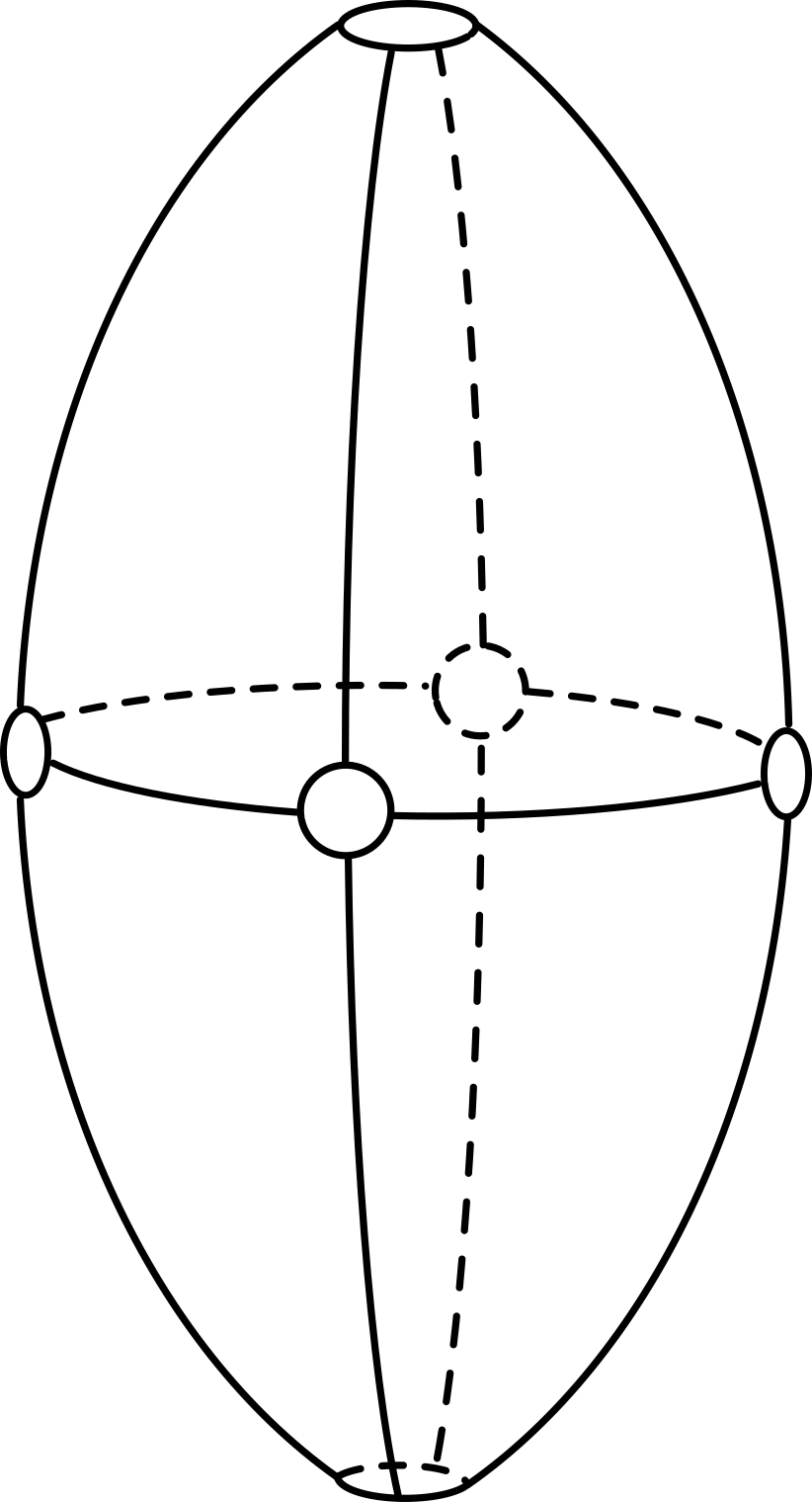}
\caption{Hollowing}
\end{subfigure}
\caption{Hollowing for the orbit space of the Hopf fibration}
\label{F: Hopf fibration}
\end{figure}
\end{enumerate}
\end{example}

\subsubsection{Decomposition of the manifold}\label{SS: decomposition of M trhough hollowings}

Let $(M,\fol)$ be   a regular foliation by circles with finite holonomy. In this case the leaf space has dimension $(n-1)$. Remark~\ref{R: non-trivial holonomy stratum circle foliation has dimension at most n-2} implies that the $(n-2)$-skeleton of the triangulation of  $M/\fol$ given in Subsection~\ref{SSS: triangulation} contains all the strata of non-trivial holonomy. As in Yano \cite{Yano1982}, this implies that the hollowing constructed in Section~\ref{S: Hollowings} is as follows:
\[
\xymatrix
{
M_{n-2}\ar[r]_{p_{n-3}}&M_{n-3}\ar[r]&...\ar[r]&M_1\ar[r]_-{p_0}&M_0=M.
}
\]
As in \cite{Fauser2019, Fauserthesis} we extend this sequence to account for the case when $M$ has non-empty boundary. We set $M_{-1} = M$, and $p_{-1} = \mathrm{Id}_{M}$.
We set $X_{-1} = \partial M \subset M_{-1}$, and define $N_j = p_{j}^{-1}(X_j)\subset M_{j+1}$. For $i<j$ we set: 
\[
p_{j,i} = p_i\circ p_{i-1}\circ \cdots \circ p_{j-1}\colon M_j\to M_i.
\]
Define $\widetilde{N}_j = p_{n-2,j+1}^{-1}(N_j)\subset M_{n-2}$, and for $\{j_1,\ldots,j_k\}$ a set of indices with $0\leq j_i\leq n-3$ we define $\widetilde{N}_{j_1,\ldots,j_k}=\widetilde{N}_{j_1}\cap\ldots\cap\widetilde{N}_{j_k}$ and ${X}_{j_1,\ldots, j_k} = p_{n-2,j_1}(\widetilde{N}_{j_1,\ldots, j_k})\subset X_{j_1}$.

We recover for the foliated context a series of lemmas from \cite{Yano1982} and \cite{Fauser2019}. For example we have the following:

\begin{lem}[Lemma~4 in \cite{Yano1982}]\label{L: X_j_1,...,j_k is contractible in orbit space}
Let $(M,\fol)$ be a regular foliation by circles with finite holonomy. For $j_1,\ldots,j_k \geq 0$, each connected component of $X_{j_1,\ldots,j_k}/\fol_{j_1}$ is contractible and can be identified with $\Delta^{j_1-k+1}_{j_1-k+1}$. 
\end{lem}

\begin{proof}
We show first that this holds for $X_{j_1}$.  The set $X_{j_1}$ is the pullback via $p_{j_1-1,0}\circ \pi \colon M_{j_1}\to M/\fol$ of the $j_1$-skeleton. When $X_{j_1}$ is connected, then from the construction, we have that $X_{j_1}/\fol_{j_1}$ is equal to $\Delta^{\ell}_{\ell}$, where $\ell=j_1$ (see Example~\ref{E: examples of hollowings}~\ref{SE: Example of hollowing of simplex}). Under this identification, each connected component of $X_{j_1,\ldots, j_k} / \fol_{j_1}$ is diffeomorphic to $\Delta^{\ell-k+1}_{\ell-k+1}$. This space is contractible.
\end{proof}

From the construction of the triangulation (Section~\ref{SSS: triangulation}) we obtain: 

\begin{lem}[Lemma~2.1 in \cite{Fauser2019}]\label{L: fibration over X_j_1,...,j_k}
Let $(M,\fol)$ be a regular foliation by circles with finite holonomy. For $j_1,\ldots,j_k \geq 0$, the space $X_{j_1,\ldots,j_k}$ is foliated diffeomorphic to $(X_{j_1,\ldots,j_k}/\fol_{j_1})\times \Sp^1$.
\end{lem}

\begin{proof}
From the construction of the triangulation, over each face of an $(n-1)$-simplex we have leaves of $\fol$ with same holonomy. By construction of $X_{j_1,\ldots,j_k}$, this fact also holds for $\fol_{j_1}$. From Lemma~\ref{L: X_j_1,...,j_k is contractible in orbit space} it follows that over a connected component of $X_{j_1,\ldots,j_k}$ we have constant holonomy. Thus $X_{j_1,\ldots,j_k}$ is the union of total spaces of circle fibrations over the connected components of $X_{j_1,\ldots,j_k} / \fol_{j_1}$. Therefore it is trivial.
\end{proof}


A main difference with the work of \cite{Fauser2019, Yano1982} is that the circle bundle $\Sp^1\to M_{n-2} \to (M_{n-2}/\fol_{n-2})$ is not trivial in general. Nonetheless, we have the following lemma:

\begin{lem}\label{L: foliation over M_n-2}
Let $(M,\fol)$ be a regular foliation by circles with finite holonomy on a connected manifold. For the last hollowing $M_{n-2}$, the leaf space $(M_{n-2}/\fol_{n-2})$ has the homotopy type of a compact connected $1$-complex and $M_{n-2}$ is aspherical.
\end{lem}

\begin{proof}
Let $\Sigma$ be the union of all strata with non-trivial holonomy. By Proposition~\ref{P: dimension of holonomy stratum}, $\Sigma$ has dimension at most $n-2$. Thus $\Sigma /\fol$ has dimension $n-3$, and is contained in simplices of dimension at most $n-3$ in $M/\fol$. From this it follows that $\bar{p}_{n-2,0}(\Sigma) = \emptyset$, since at this point we have removed all the preimages of the simplices of dimension less than $n-2$. Thus over $M_{n-2}$ the foliation does not have holonomy, and induces a circle bundle. Observe that $M_{n-2}/\fol_{n-2} \cong M/\fol \setminus (M/\fol)^{(n-3)}$ has the homotopy type of a compact $1$-complex, that is a finite graph. Since $M$ is connected, this graph is connected. The fundamental group of a finite graph is a finitely generated free group. From the long exact sequence of homotopy groups of the fibration $\Sp^1\to M_{n-2}\to M_{n-2}/\fol_{n-2}$ we see that $M_{n-2}$ is aspherical.
\end{proof}



Furthermore, by the construction of the hollowings and the triangulation, the following proposition holds:

\begin{prop}[See Proposition~2.2 in \cite{Fauser2019}]\label{P: Connected components of X_j_1,...,j_k,-1}
For all pairwise distinct $j_1,\ldots,j_k\in \{0,\ldots,n-q-2\}$ we have that $X_{j_1,\ldots,j_k,-1}$ is the union of the connected components $Y\subset X_{j_1,\ldots,j_k}$ that satisfy
\[
	Y\subset p_{n-2,j_1}(\widetilde{N}_{-1}).
\] 
\end{prop}
\begin{proof}
Let $j\in \{0,\ldots, n-3\}$. We first show the statement for $X_{j,-1}\subset X_j$. We work in the space $X_j/\fol_j$ and show the statement there.

Let $Y\subset X_j/\fol_j$ be a connected component. We will show that \[
Y\subset X_{j, -1}/\fol_j \Longleftrightarrow Y\subset p_{n-2,j}(\widetilde{N}_{-1})/\fol_j.
\]
The right implication is true by definition. For the left one, as in Yano’s proof of Lemma 2.1 \cite{Yano1982}, we observe that $Y$ is homeomorphic to $\Delta_j^j$, where $\Delta_j^j$ is obtained from the standard simplex $\Delta^j$ by hollowing inductively along the $\ell$-skeleton for all $\ell\in \{0, \ldots,j-1\}$ (see Example \ref{E: examples of hollowings} \ref{SE: Example of hollowing of simplex} above). From this it follows easily that we are in one of the following cases:
\begin{enumerate}
\item $Y\subset p_{n-2,j}(\widetilde{N}_{-1})/\fol_j$, or
\item $Y\cap p_{n-2,j}(\widetilde{N}_{-1})/\fol_j =\emptyset$.
\end{enumerate}
In the first case, we have
\[
\begin{array}{llr}
Y&\subset& X_j/\fol_j\cap p_{n-2,j}(\widetilde{N}_{-1})/\fol_j\\
& =& p_{n-2,j}(\widetilde{N}_j)/\fol_j\cap p_{n-2,j}(\widetilde{N}_{-1})/\fol_j\\
 &\subset& p_{n-2,j}(\widetilde{N}_{j,-1})/\fol_j\\
 & =&X_{j,-1}/\fol_j,
 \end{array}
 \]
where the last inclusion follows from 
\[
p_{n-2,j}(\widetilde{N}_j\setminus \widetilde{N}_{-1})/\fol_j\cap p_{n-2,j}(\widetilde{N}_{-1}\setminus \widetilde{N}_j)/\fol_j =\emptyset,
\]
which holds by construction of the hollowings. In the second case, we have $Y\cap X_{j,-1}/\fol_j=\emptyset$.

For $X_{j_1,\ldots,j_k,-1}$ with $k\geq 1$, it suffices to observe that 
\[
X_{j_1, \ldots,j_k,-1}/\fol_j=X_{j_1, \ldots, j_k}/\fol_j\cap X_{j_1, -1}/\fol_j.
\]
See also \cite[Lemma 4.2.8]{Fauserthesis} for more details.
\end{proof}

\begin{remark}\label{R: triangulation boundary}
We point out that we can take a refinement of the triangulation on $\partial (M_{n-2}/\fol_{n-2})$ so that  it is compatible with the decompositions
\begin{align*}
\partial  (M_{n-2}/\fol_{n-2})  &= \bigcup^{n-3}_{i=-1} \widetilde{N}_i/\fol_{n-2}
\end{align*} and
\begin{align*}
\partial   \left( \widetilde{N}_{i_1, \ldots, i_k} / \fol_{n-2}\right) &= \bigcup^{n-3}_{i=-1, i\neq i_1, \ldots, i_k} \widetilde{N}_{i_1, \ldots, i_k, i}/ \fol_{n-2}.
\end{align*}			
That is, each $\widetilde{N}_{i_1, \ldots, i_k, i}/\fol_{n-2}$ is a subcomplex of $ \partial (M_{n-2}/\fol_{n-2})$.
\end{remark}
\section{Proof of Theorem B}\label{S: Proof}

Now we establish the necessary preliminary results for the proof of our main theorem, which is carried out at the end of the section.

We will then use the triangulation of $M/\fol$ to construct a series of triangulations on the holonomy strata that have zero foliated simplicial volume.

%


\begin{prop}[See Proposition~4.1 in \cite{Fauser2019}]\label{pi1-injective}
Assume the inclusions of the leaves of the foliation $\fol$ are $\pi_1$-injective. Take $k \in \{1,\ldots, n-2\}$ and let $j_1,\ldots,j_k \in \{0,\ldots,n-3\}$ be pairwise distinct. Then, for any choice of basepoints, the inclusions $X_{j_1,\ldots,j_k}\subset M_{j_1}$ and $X_{j_1,\ldots,j_k,-1} \subset M_{j_1}$ are $\pi_1$-injective.
\end{prop}
\begin{proof}
By Proposition \ref{P: Connected components of X_j_1,...,j_k,-1}, it suffices to show that the inclusion $X_{j_1,\ldots,j_k}\subset M_{j_1}$ is $\pi_1$-injective. By Lemma \ref{L: fibration over X_j_1,...,j_k}, we have $X_{j_1,\ldots,j_k}\cong (X_{j_1,\ldots,j_k}/\fol_{j_1})\times \Sp^1$. By Lemma \ref{L: X_j_1,...,j_k is contractible in orbit space}, each connected component of $X_{j_1,\ldots,j_k}/\fol_{j_1}$ is contractible. Now, by Proposition \ref{P: The hollowing admits a foliation}, the composition of maps 
\[
X_{j_1,\ldots,j_k}\subset M_{j_1}\stackrel{p_{j_1, 0}}{\longrightarrow} M_0=M
\]
is the inclusion of leaves into $M$, 
and thus $\pi_1$-injective by hypothesis. Thus the inclusion $X_{j_1,\ldots,j_k}\subset M_{j_1}$ is also $\pi_1$-injective.
\end{proof}
We now construct a series of representations of the fundamental groups of the hollowings $M_j$ as follows:
\begin{setup}\label{Setup: restrictions of representations}
Fix $x_{n-2}\in M_{n-2}$, and set $x_i = p_{n-2,i}(x_{n-2})\in M_i$. We write $\Gamma = \pi_1(M,x_0)$ and consider a fixed essentially free  standard $\Gamma$-space $(Z,\mu)$, with the representation $\alpha_0 = \alpha\colon \Gamma\to \mathrm{Aut}(Z,\mu)$. From this representation, using the hollowing maps $p_{i,0}\colon M_i\to M$, we can define for $\Gamma_i = \pi_1 (M_i,x_i)$ a representation $\alpha_i\colon \Gamma_i\to \mathrm{Aut}(Z,\mu)$ by setting
\[
	\alpha_i = \alpha\circ \pi_1(p_{i,0}).
\]
Recall that we have a circle bundle $M_{n-2}\to M_{n-2}/\fol_{n-2}$ (see proof of Lemma~\ref{L: foliation over M_n-2}), which may be orientable or not. If the bundle structure of $M_{n-2}$ is orientable, then by the classification of oriented $\Sp^1$-bundles, it is trivial: by Lemma~\ref{L: foliation over M_n-2} $M_{n-2}/\fol_{n-2}$ has the homotopy type of a graph, and hence $H^2(M_{n-2}/\fol_{n-2}, \mathbb{Z})=0$. 

If the bundle structure of $M_{n-2}$ is non-orientable, since $M_{n-2}$ is orientable, then the base $M_{n-2}/\fol_{n-2}$ is non-orientable. For this case we need to change the Borel space we consider as follows: we take the oriented double cover of $M_{n-2}/\fol_{n-2}$. By pulling back the circle bundle we obtain an orientable double cover $W$ of $M_{n-2}$. This double cover is homotopy equivalent to an oriented circle bundle over 
 the double cover of $M_{n-2}/\fol_{n-2}$, denoted by $B$. Thus we get the following commutative diagram:


$$
\xymatrix
{
W \ar[r]^{\tilde{p}}_2\ar[d] & M_{n-2}\ar[d]\\
B \ar[r]^-p_-2 & M_{n-2}/\fol_{n-2}
}
$$
Observe that the fundamental group $H$ of $W$ is a subgroup of index $2$ of $\Gamma_{n-2}$, and thus we obtain a representation $\beta$ of $H$ on $(Z,\mu)$ by restricting $\alpha_{n-2}$. In the subsequent proofs, we will find parametrized relative fundamental cycles of $W$, and from them we will obtain appropriate parametrized fundamental cycles of $M_{n-2}$.  

To do so, as in \cite[Setup~4.24 and Definition~4.25]{LoehPagliantini2016}, set $\gamma_0= e$ and $\gamma_1$ be a fixed representative of the non-identity class in $\Gamma_{n-2}/H$. We denote the elements in
\[\Gamma_{n-2}\times_H Z=\Gamma_{n-2}\times Z/\{(\gamma h, z)\sim (\gamma, h\cdot z), \gamma\in\Gamma_{n-2}, h\in H, z\in Z\}\]
by $[\gamma,z]$, and $\Gamma_{n-2}$ acts on $\Gamma_{n-2}\times_H Z$ by $\gamma'[\gamma,z] = [\gamma'\gamma,z]$. The measure on $\Gamma_{n-2}\times_{H} Z$ is given as follows: we take the counting measure $\mu'$ on $\Gamma_{n-2}/ H$ and then pull back the measure $(1/2)\mu'\otimes \mu$ on $\Gamma_{n-2}/H \times Z$ via the bijection
\begin{align*}
\Gamma_{n-2}\times_{H} Z &\to \Gamma_{n-2}/H \times Z,\\
							[\gamma,z] & \mapsto (\gamma H,z).
\end{align*}
As in \cite[Proof of Proposition~4.26]{LoehPagliantini2016}, we have the following well-defined $\Z\Gamma_{n-2}$-isomorphism:
\begin{align*}
	\psi\colon L^\infty(Z,\Z)\otimes_{\Z H} \Z \Gamma_{n-2}&\longrightarrow L^\infty(\Gamma_{n-2}\times_H Z,\Z)\\[0.5em]
	f\otimes \gamma_j &\longmapsto \left([\gamma_k,z] \mapsto\begin{cases} f(z), & \mbox{if } k=j\\
	0, & \mbox{if } k\neq j
	\end{cases} \right).
\end{align*}
The induced map
\begin{align*}
 \Psi\colon C_{\ast}(W; \beta)&\longrightarrow C_{\ast}(M_{n-2}; \Gamma_{n-2}\times_H Z)\\[0.5em]
 f\otimes \sigma &\longmapsto \psi(f\otimes e)\otimes \sigma
\end{align*}
sends parametrized fundamental cycles of $W$ to parametrized fundamental cycles of $M_{n-2}$ \cite[Proof of Proposition~4.26]{LoehPagliantini2016}. The parametrized norms behave as follows:
\begin{equation}\label{E: parametrized volumes of double cover}
\oldnorm{\Psi(c)}^{\Gamma_{n-2}\times_H Z}\leq \frac{1}{2}\oldnorm{c}^Z
\end{equation}
for all $c\in C_{\ast}(W; \beta)$. 

Using the hollowings $p_i\colon M_{i+1}\to M_i$, we define $H_i = \pi_1(p_{n-2,i}\circ \tilde{p})(H) < \Gamma_i$. Note that $H_i$ has finite index in $\Gamma_i$ for every $i\in\{0, \ldots, n-2\}$. Indeed, the maps $\pi_1(p_{n-2,i})$ induced by the hollowings are surjective: the hollowing maps $p_i$ are quotient maps by construction, and their fibers are either a point, or a sphere of dimension at least $1$, hence connected. Then \cite[Theorem 1.1]{CalcutGompfMcCarthy2012} applies. Moreover, the index of $H$ in $\Gamma_{n-2}$ is $2$ by definition, so that $[\Gamma_i: H_i]\leq 2$.

We note that by construction of the representations $\alpha_i$, the restriction of $\alpha_i$ to $H_i$ is an essentially free action on $Z$.
The spaces we consider are $\Gamma_i\times_{H_i} Z$, which have an essentially free action of $\Gamma_i$ as above, denoted by $\beta_i$.

Depending on whether the circle bundle $M_{n-2}\to M_{n-2}/\fol_{n-2}$ is orientable or not,  for $i\in \{0,\ldots,n-2\}$, we consider $V_i$ equal to $Z$, respectively $\Gamma_i\times_{H_i} Z$, with a representation $\xi_i$ of $\Gamma_i$ given by $\alpha_i$, respectively $\beta_i$.

Let $\widetilde{p}_i\colon \widetilde{M}_{i+1}\to \widetilde{M}_i$ be a fixed lift to the universal covers of the map $p_i\colon M_{i+1}\to M_i$. 
We define a chain map 
\begin{align*}
P_{i}\colon L^\infty(\xi_{i+1},\Z)\otimes_{\Z \Gamma_{i+1}} C_\ast (\widetilde{M}_{i+1},\Z) &\longrightarrow L^\infty(\xi_{i},\Z)\otimes_{\Z \Gamma_{i}} C_\ast (\widetilde{M}_{i},\Z)\\[0.5em]
f\otimes \sigma &\longmapsto f\otimes (\widetilde{p}_i \circ\sigma).
\end{align*}
In this way we obtain the following sequence:
\[
	L^\infty(\xi_{n-2},\Z)\otimes_{\Z \Gamma_{n-2}} C_\ast (\widetilde{M}_{n-2},\Z)\overset{P_{n-3}}{\longrightarrow}\cdots \overset{P_{0}}{\longrightarrow} L^\infty(\xi_{0}, \Z)\otimes_{\Z \Gamma_{0}} C_\ast (\widetilde{M},\Z).
\]
For $i<j$ we can also consider the maps 
\[
P_{j,i} \colon L^\infty(\xi_{j},\Z)\otimes_{\Z \Gamma_{j}} C_\ast (\widetilde{M}_{j},\Z) \longrightarrow L^\infty(\xi_{i},\Z)\otimes_{\Z \Gamma_{i}} C_\ast (\widetilde{M}_{i},\Z),
\]
defined as $P_{j,i} = P_{i} \circ \cdots \circ P_{j-1}$. These maps will be used later in the proof of Theorem~\ref{Theorem: circle foliation implies the vanishing of foliated simplicial volume}.

For $X_{i_1,\ldots, i_k}\subset M_{i_1}$ we set $\Lambda_{i_1,\ldots, i_k}  = \pi_1(X_{i_1,\ldots, i_k})$. By the observation made in \cite[Setup~4.2]{Fauser2019}, these groups are independent of the base points chosen. By Proposition \ref{pi1-injective}, we have $\Lambda_{i_1,\ldots, i_k} < \Gamma_{i_1}$. We denote by $\xi'_{i_1,\ldots, i_k}$ the restriction of the representation $\xi_{i_1}$ to $  \Lambda_{i_1,\ldots,i_k} $. For the universal cover $q_{i_1}\colon \widetilde{M}_{i_1} \to M_{i_1}$, observe that $q_{i_1}^{-1}(X_{i_1, \ldots, i_k})$ is $\Gamma_{i_1}$-invariant. Hence we can consider the subcomplex 
\[
	L^\infty(\xi_{i_1},\Z)\otimes_{\Z \Gamma_{i_1}} C_\ast (q_{i_1}^{-1}(X_{i_1,\cdots,i_k}),\Z).
\]
\end{setup}
\begin{prop}
For the subcomplex $L^\infty(\xi_{i_1};\Z)\otimes_{\Z \Gamma_{i_1}} C_\ast (q_{i_1}^{-1}(X_{i_1,\cdots,i_k}),\Z)$ and the restriction $\xi'_{i_1,\cdots,i_k}$ we have an isomorphism from
\[
L^\infty(\xi_{i_1},\Z)\otimes_{\Z \Gamma_{i_1}} C_\ast (q_{i_1}^{-1}(X_{i_1,\cdots,i_k}),\Z) 
\]
 onto 
\[
 L^\infty(\xi'_{i_1,\cdots,i_k},\Z)\otimes_{\Z \Lambda_{i_1,\cdots,i_k}} C_\ast (X_{i_1,\cdots,i_k},\Z).
 \]
\end{prop}

\begin{proof}
See \cite[p. 12]{Fauser2019}.
\end{proof}

For  a fixed $\varepsilon>0$, we will show the existence of an essentially free $\Gamma$-space, and a representation of $\Gamma$, such that there is  a relative parametrized fundamental cycle of $M$ with $\ell^1$-norm bounded above by $\varepsilon$. We begin by finding such a cycle for $M_{n-2}$.

\begin{prop}\label{P: fundamental cycle of regular part}
Let $(M,\fol)$ be an oriented compact connected smooth $n$-manifold with a regular foliation by circles with finite holonomy. Assume that the inclusion of each  leaf into $M$ is $\pi_1$-injective. Set $\Gamma = \pi_1(M,x_0)$ and choose $\varepsilon >0$. There exists a relative fundamental cycle
	\[
		z\in C_n(M_{n-2};\xi_{n-2})
	\]
that has $\ell^1$-norm less than $\varepsilon$.
\end{prop}

\begin{proof}
We consider two cases: when $M_{n-2}/\fol_{n-2}$ is orientable and when it is not. In the first case, since $M_{n-2}$ is orientable, then the circle bundle $M_{n-2}\to M_{n-2}/\fol_{n-2}$ is also orientable. Thus, as stated in the proof of Lemma~\ref{L: foliation over M_n-2}, we have $M_{n-2} \cong \left(M_{n-2}/\fol_{n-2}\right) \times \Sp^1$. Observe that, by Proposition \ref{P: The hollowing admits a foliation}, each leaf in $M_{n-2}$ is mapped by $p_{n-3,0}$ to a leaf of $(M,\fol)$. The subgroup $\Lambda$ of $\pi_1(M_{n-2}/\fol_{n-2})\times \pi_1(\Sp^1)$ generated by the circle factor corresponds under this homeomorphism  to the subgroup of $\Gamma_{n-2}$ generated by a leaf. Since the inclusion of any leaf is $\pi_1$-injective, then $(Z,\mu)$ is an essentially free standard $\Lambda$-space with respect to $\alpha'$, the restriction of $\alpha_{n-2}$ to $\Lambda$. From the proof of Lemma~10.8 in \cite{FauserLoeh2019}, given any relative fundamental cycle $\bar{z}$ of $M_{n-2}/\fol_{n-2}$, there exists a cycle $c_{\Sp^1}\in C_1(\Sp^1;\alpha')$ such that the relative fundamental cycle
\[
	z = \bar{z}\times c_{\Sp^1} \in C_n(M_{n-2};\alpha_{n-2}),
\]
has $\ell^1$-norm less than $\varepsilon$, as desired.

For the second case,  we consider the oriented double cover $\tilde{p}\colon W \to M_{n-2}$ of $M_{n-2}$, which is the total space of a trivial circle bundle over the orientable double cover $B$ of $M_{n-2}/\fol_{n-2}$.  Recall that $Z$ is an essentially free $H$-space, via the representation $\beta$.
As in the first case, since $B$ has the homotopy type of a $1$-complex, the fundamental group of the fiber $\Sp^1$ injects into $H$. Denote  by $\beta'$ the restriction of $\beta$ to this subgroup. Again by the proof of Lemma~10.8 in \cite{FauserLoeh2019}, for any relative fundamental cycle $\bar{u}$ of $B$, we can find a parametrized  cycle $c_{\Sp^1}\in C_1(\Sp^1;\beta')$ such that $\bar{u}\times c_{\Sp^1}$ has $\ell^1$-norm less than $2\varepsilon$.

Recall that the parametrized norms behave as follows
\[
\oldnorm{\Psi(c)}^{\Gamma_{n-2}\times_H Z}\leq \frac{1}{2}\oldnorm{c}^Z
\]
for all $c\in C_{\ast}(W; \beta)$. 
Hence, taking $c=\bar{u}\times c_{\Sp^1}\in C_{n}(W;\beta)$, we obtain a $\Gamma_{n-2}\times_H Z$-parametrized relative fundamental cycle $z=\Psi(c)$ for $M_{n-2}$ with 
\[
\oldnorm{z}^{\Gamma_{n-2}\times_H Z}< \varepsilon.
\]
\end{proof}


\begin{remark}
The hypothesis requiring that the inclusion of any  leaf is $\pi_1$-injective is too strong: it is sufficient that the action $\xi_{n-2}$ restricted to the image subgroup of the fundamental group of a  leaf remains essentially free. That is that for the action of $\Lambda$  on $Z$, the set of elements $x\in Z$ which have non-trivial isotropy has zero measure.
\end{remark}

We will now show the existence of a fundamental cycle of $(M,\partial M)$ with arbitrarily small $\ell^1$-norm. We first consider the case when the circle bundle $M_{n-2}\to M_{n-2}/\fol_{n-2}$ is orientable. We fix $\varepsilon>0$ and consider  $z = \bar{z}\times c_{\Sp^1}\in C_n(M_{n-2};\alpha_{n-2})$, the cycle obtained from Proposition~\ref{P: fundamental cycle of regular part}.  
For $n-3\geqslant i\geqslant -1$, we define a cycle $\bar{z}_i\in C_{n-2}(\widetilde{N}_i/\fol_{n-2};\Z)$ as the sum of all the simplices in $\partial \bar{z}$ that belong to the subcomplex $\widetilde{N}_i/\fol_{n-2} \subset \partial M_{n-2}/\fol_{n-2}$. We define
\[
z_i := \bar{z}_i \times c_{\Sp^1}\in C_{n-1}(\widetilde{N}_i;\alpha_{n-2}).
\]
In an analogous fashion,  for a subset of pairwise distinct indices $i_1,\ldots,i_k$ with $n-3\geqslant i_j \geqslant -1$,  we define cycles $\bar{z}_{i_1,\ldots,i_k}\in C_{n-1-k}(\widetilde{N}_{i_1,\ldots,i_k}/\fol_{n-2};\Z)$ as the sum of all the simplices of $\partial \bar{z}_{i_1,\ldots,i_{k-1}}$ contained in $\widetilde{N}_{i_1,\ldots,i_k}/\fol_{n-2}$. For non-pairwise distinct indices we set $\bar{z}_{i_1,\ldots,i_k}=0$. We define
\[
	z_{i_1,\ldots,i_k} := \bar{z}_{i_1,\ldots,i_k} \times c_{\Sp^1}\in C_{n-k}(\widetilde{N}_{i_1,\ldots,i_k};\alpha_{n-2}).
\]
When the circle bundle $M_{n-2}\to M_{n-2}/\fol_{n-2}$ is not orientable, we use its oriented double cover. 
Let $W\rightarrow B$ be the oriented double cover of $M_{n-2}\to M_{n-2}/\fol_{n-2}$, as in the Setup~\ref{Setup: restrictions of representations}:
$$
\xymatrix
{
W \ar[r]^{\tilde{p}}_2\ar[d] & M_{n-2}\ar[d]\\
B \ar[r]^-p_-2 & M_{n-2}/\fol_{n-2}
}
$$
We have corresponding preimages
$$
\xymatrix
{
\widetilde{W}\ar[r]\ar[d]_q& q_{n-2}^{-1}(\widetilde{N}_j)\subset \widetilde{M}_{n-2}\ar[d]_{q_{n-2}}\\
\tilde{p}^{-1}(\widetilde{N}_j)\subset W \ar[r]^{\tilde{p}}\ar[d] & \widetilde{N}_j\subset M_{n-2}\ar[d]\\
p^{-1}(\widetilde{N}_j/\fol_{n-2})\subset B \ar[r]^-p & \widetilde{N}_j/\fol_{n-2}\subset M_{n-2}/\fol_{n-2}
}
$$
We recall that $W\rightarrow B$ is a trivial circle bundle. Proposition \ref{P: fundamental cycle of regular part} thus gives us a parametrized relative fundamental cycle of the form $u = \bar{u}\times c_{\Sp^1} \in C_n(W; \beta)$, where $\bar{u}$ is any relative fundamental cycle of $B$.


For $n-3\geqslant i\geqslant -1$, we define a cycle $\bar{u}_i\in C_{n-2}(p^{-1}(\widetilde{N}_i/\fol_{n-2}), \Z)$ as the sum of all the simplices in $\partial \bar{u}$ that belong to the subcomplex $p^{-1}(\widetilde{N}_i/\fol_{n-2}) \subset \partial B$. We define
\[
u_i := \bar{u}_i \times c_{\Sp^1}\in C_{n-1}(\tilde{p}^{-1}(\widetilde{N}_i);\beta).
\]
In an analogous fashion,  for a subset of pairwise distinct indices $i_1,\ldots,i_k$ with $n-3\geqslant i_j \geqslant -1$,  we define cycles
\[
\bar{u}_{i_1,\ldots,i_k}\in C_{n-1-k}(p^{-1}(\widetilde{N}_{i_1,\ldots,i_k}/\fol_{n-2}),\Z),
\]
as the sum of all the simplices of $\partial \bar{u}_{i_1,\ldots,i_{k-1}}$ contained in $p^{-1}(\widetilde{N}_{i_1,\ldots,i_k}/\fol_{n-2})$. For non-pairwise distinct indices we set $\bar{u}_{i_1,\ldots,i_k}=0$. We define
\[
	u_{i_1,\ldots,i_k} := \bar{u}_{i_1,\ldots,i_k} \times c_{\Sp^1}\in C_{n-k}(\tilde{p}^{-1}(\widetilde{N}_{i_1,\ldots,i_k});\beta).
\]
Using the map $\Psi$ introduced in the proof of Proposition~\ref{P: fundamental cycle of regular part}, we write
\[
\begin{array}{lll}
z&=&\Psi(u)\in C_n(M_{n-2};\beta_{n-2}),\\
z_{i_1, \ldots, i_k}&=& \Psi(u_{i_1,\ldots,i_k})\in C_{n-k}(\widetilde{N}_{i_1,\ldots,i_k};\beta_{n-2}).
\end{array}
\]

With this notation, we have the following three lemmas.
\begin{lem}[See \cite{Fauser2019}, Lemma 6.1]\label{L: boundary of z, z_i}
We have
\[\partial z=\sum_{i=-1}^{n-3}z_i \,\mbox{ and } \,\partial z_{i_1, \ldots, i_k}=\sum_{i=-1}^{n-3}z_{i_1, \ldots, i_k, i}\]
for all $k\in\{1, \ldots, n-1\}$ and pairwise distinct $i_1, \ldots, i_k\in \{-1, ..., n-3\}$.
\end{lem}

\begin{proof}
In the orientable case, by definition of $z$ and $z_{i_1, \ldots, i_k}$, it is enough to show the analogous statements for $\bar{z}$ and $\bar{z}_{i_1, \ldots, i_k}$.

Recall from Remark~\ref{R: triangulation boundary} that the boundary  $\partial(M_{n-2}/\fol_{n-2})$ is a union of  subcomplexes $\cup^{n-3}_{i=-1} \widetilde{N}_i^\ast$ of the simplicial structure on $\partial(M_{n-2}/\fol_{n-2})$. It follows that
\[
\partial\bar{z}=\partial\bar{z}|_{\partial (M_{n-2}/\fol_{n-2})}=\sum_{i=-1}^{n-3}\partial\bar{z}|_{\widetilde{N}_{i}^\ast}=\sum_{i=-1}^{n-3}\bar{z}_i.
\]
Moreover, for all $k\in\{1, \ldots, n-1\}$ and all pairwise distinct $i_1, \ldots, i_k\in \{-1, ..., n-3\}$, we have that
\begin{align*}
\partial\bar{z}_{i_1, \ldots, i_k} &=\partial\bar{z}_{i_1, \ldots, i_k}|_{\partial (\widetilde{N}_{i_1, \ldots, i_k}/\fol_{n-2})}\\[0.8em]
&=\sum_{i\neq i_1, \ldots, i_k}\partial\bar{z}_{i_1, \ldots, i_k}|_{\widetilde{N}_{i_1, \ldots, i_k, i}/\fol_{n-2}}\\[0.8em]
&=\sum_{i=-1}^{n-3}\bar{z}_{i_1, \ldots, i_k, i}.
\end{align*}
In the non-orientable case, we compute:
\[
\partial\bar{z}=\partial\Psi(u)=\Psi(\partial(u))=\Psi(\partial(\bar{u})\times c_{\Sp^1}).
\]
Then we remark
\[
\partial\bar{u}=\partial\bar{u}|_{\partial (p^{-1}(M_{n-2}/\fol_{n-2}))}=\sum_{i=-1}^{n-3}\partial\bar{u}|_{p^{-1}(\widetilde{N}_{i}^\ast)}=\sum_{i=-1}^{n-3}\bar{u}_i.
\]
We insert it in the previous computation and obtain the conclusion.

An analogous reasoning shows also the formula for $\partial z_{i_1, \ldots, i_k}$.
\end{proof}
\begin{lem}[See \cite{Fauser2019}, Lemma 6.2]\label{L: alternation}
Let $k\in \{1, \ldots, n-1\}$ and let $\tau\in \mathrm{Sym}(k)$ be a permutation of $\{1, \ldots, k\}$. Then we have
\[
z_{i_1, \ldots, i_k}=\mathrm{sign}(\tau)z_{i_{\tau(1)}, \ldots, i_{\tau(k)}}.
\]
\end{lem}
\begin{proof}
We may assume that $\tau$ is a transposition $(i_j, i_{j+1})$. By definition of $\widetilde{N}_{i_1,\ldots,i_k}$ and $z_{i_1, \ldots, i_k}$, we may even assume that $\tau=(i_{k-1}, i_k)$. Thus we have to show that 
\[
z_{i_1, \ldots,i_{k-1}, i_k}=-z_{i_1, \ldots, i_{k-2}, i_k, i_{k-1}}.
\]
By Lemma \ref{L: boundary of z, z_i}, we have
\[
0=\partial\partial z_{i_1, \ldots,i_{k-2}}=\partial\left(\sum_{i=-1}^{n-3} z_{i_1, \ldots,i_{k-2}, i}\right)=\sum_{j=-1}^{n-3}\sum_{i=-1}^{n-3}z_{i_1, \ldots,i_{k-2}, i, j}.
\]
In the orientable case, since $\partial(\widetilde{N}_{i_1, ..., i_{k-2}}/\fol_{n-2})$ is a subcomplex of the simplicial structure on $M_{n-2}/\fol_{n-2}$, and from the definition of $z_{i_1, \ldots,i_{k-2}}$,
it follows that cancellations may occur only between terms with the same set of indices. Hence the only possibility is
\[z_{i_1, \ldots,i_{k-1}, i_k}=-z_{i_1, \ldots, i_{k-2}, i_k, i_{k-1}}.\]
In the non-orientable case, the exact same relations hold for the $u_{i_1, \ldots, i_k}$. Then apply the chain map $\Psi$ to finish the proof.
\end{proof}
\begin{lem}[See \cite{Fauser2019}, Lemma 6.3]\label{L: efficient fillings}
There exist chains 
\[
	w_{i_1,\ldots,i_k}\in C_{n-k+1}(X_{i_1,\ldots,i_k};\xi_{i_1,\ldots,i_k}'),
\]
\[
	w_{i_1,\ldots,i_k,-1}\in C_{n-k}(X_{i_1,\ldots,i_{k},-1};\xi_{i_1,\ldots,i_k,-1}'),
\]
with $k\in\{1, \ldots, n-2\}$ and $i_1, ..., i_k\in\{0, \ldots, n-3\}$, and a constant $C\in \R_+$ depending only on $k$, such that
\begin{enumerate}[(i)]
\item the chains $w_{i_1,\ldots,i_k}$ and $w_{i_1,\ldots,i_{k},-1}$ are alternating with respect to permutations of the indices $\{i_1, \ldots, i_k\}$;
\item the following relations hold:
\begin{align*}
\bullet\; \partial w_{i_1,\ldots,i_k} &= P_{n-2,i_1}(z_{i_1,\ldots,i_k}) - \sum^{n-3}_{i=-1} w_{i_1,\ldots,i_k,i},\\[0.8em]
\bullet\;\partial w_{i_1,\ldots,i_{n-2},-1} &= P_{n-2,i_1}( z_{i_1,\ldots,i_{n-2},-1});
\end{align*}
\item  $|w_{i_1,\ldots,i_k}|_1 \leqslant C |z|_1$. The index $i_k$ is allowed to take the value $-1$.
\end{enumerate}
\end{lem}

\begin{proof}
Recall that $X_{i_1,\ldots,i_k} \cong (X_{i_1,\ldots,i_k} / \fol_{i_1}) \times \Sp^1$, and that for both the cases of orientability of the circle bundle $M_{n-2}\to M_{n-2}/\fol_{n-2}$ we have a series of essentially free representations of the fundamental groups of the hollowings $M_i$.  To prove this Lemma, we apply  Lemmas \ref{L: boundary of z, z_i}, \ref{L: alternation} above to the proof of \cite[Lemma~6.3]{Fauser2019} for the essentially free representations $\alpha_i$.
%
\end{proof}

Now we are ready to prove the main theorem of the present note. 

\begin{proof}[\bfseries{Proof 
of Theorem~\ref{Theorem: circle foliation implies the vanishing of foliated simplicial volume}}]
Given $\varepsilon>0$, from Proposition~\ref{P: fundamental cycle of regular part}, there exists a (relative) fundamental cycle $z\in C_{n}(M_{n-2};\xi_{n-2})$ with $\ell^1$-norm less than $\varepsilon$. From Lemma~\ref{L: efficient fillings}, we have associated chains $w_{i}\in C_{n}(X_i;\xi'_{i})$ with $\ell^1$-norm less than $C|z|_1$, where $C$ is the constant of Lemma~\ref{L: efficient fillings}.  We claim that the chain
\[
	z':= P_{n-2,0}(z)-\sum^{n-3}_{i=0} P_{i,0}(w_i)\in C_{n}(M;\xi)
\]
is a $\xi$-parametrized relative fundamental cycle of $M$. From \cite[Proposition~3.9]{Fauser2019}, it is sufficient to show that $z'$ is a $U$-local $\xi$-parametrised relative fundamental cycle of $M$, for an arbitrary open subset $U\subset M\setminus p_{n-2,0}(\partial M_{n-2})$ diffeomorphic to an $n$-disk.  Since the boundary of $M_{n-2}$ contains all the preimages of the hollowings, we have $p_{n-2}(z) = z'$ on any small ball in $M\setminus p_{n-2,0}(\partial M_{n-2})$. The rest of the proof of the claim follows from the proof of \cite[Theorem~1.1]{Fauser2019}. 

We recall that for $f\otimes \sigma \in C_{n}(M_{n-2};\xi_{n-2})$ we have 
\[
	P_{n-2, 0}(f\otimes \sigma) = f\otimes \widetilde{p}_{0}\circ\cdots \circ \widetilde{p}_{n-3}(\sigma)\in C_n(M; \xi).
\]
Thus we have that
\[
	|z'|_1\leqslant |z|_1+\sum_{i=0}^{n-3}|w_i|_1 < \varepsilon,
\]
when we choose $z\in C_{n}(M_{n-2};\xi_{n-2})$ such that $|z|_1 <\frac{\varepsilon}{(n-2)C+1}$.
\end{proof}

\begin{remark}
In the case when the circle bundle $M_{n-2}\to M_{n-2}/\fol_{n-2}$ is orientable, the parametrized norm vanishes for arbitrary essentially free $\Gamma$-spaces. In particular, when $\Gamma$ is residually finite, its profinite completion is such an essentially free $\Gamma$-space. By \cite[Theorem 2.6]{frigerio-loeh-pagliantini-sauer}, this shows that the stable integral simplicial volume of $M$ also vanishes.

If the bundle $M_{n-2}\to M_{n-2}/\fol_{n-2}$ is not orientable, our proof shows that starting from any measured $\Gamma$-space $Z$, we find arbitrarily small fundamental cycles parametrized by the space $\Gamma\times_{H_0} Z$ with the action $\beta_0$, where $H_0=\pi_1( p_{n-2, 0}\circ\tilde{p})(H)$. 
This however does not allow us a priori to conclude about the action on the profinite completion of $\Gamma$, except if $H_0=\Gamma$. In this case, the space $\Gamma\times_{H_0} Z$ can be identified with $Z$;  the action $\beta_0$ on $\Gamma\times_{H_0} Z$ corresponds under this identification to the original action $\gamma\cdot z=\alpha(\gamma)(z)$ of $\Gamma$ on $Z$.
Thus, setting $Z$ to be the profinite completion of $\Gamma$ implies again, by \cite[Theorem 2.6]{frigerio-loeh-pagliantini-sauer}, that the stable integral simplicial volume of $M$ vanishes.
\end{remark}

%
%
%
%
%
%
%
%

\bibliographystyle{siam}
\bibliography{Tesis_Doc}

\begin{thebibliography}{10}

\bibitem{Alexandrino}
{\sc M.~M. Alexandrino and R.~G. Bettiol}, {\em Lie groups and geometric
  aspects of isometric actions}, Springer, Cham, 2015.

\bibitem{CalcutGompfMcCarthy2012}
{\sc J.~S. Calcut, R.~E. Gompf, and J.~D. McCarthy}, {\em On fundamental groups
  of quotient spaces}, Topology Appl., 159 (2012), pp.~322--330.

\bibitem{CandelConlon2000}
{\sc A.~Candel and L.~Conlon}, {\em Foliations. {I}}, vol.~23 of Graduate
  Studies in Mathematics, American Mathematical Society, Providence, RI, 2000.

\bibitem{CheegerGromov1986}
{\sc J.~Cheeger and M.~Gromov}, {\em Collapsing {R}iemannian manifolds while
  keeping their curvature bounded. {I}}, J. Differential Geom., 23 (1986),
  pp.~309--346.

\bibitem{EdmondsFintushel1976}
{\sc A.~L. Edmonds and R.~Fintushel}, {\em Singular circle fiberings}, Math.
  Z., 151 (1976), pp.~89--99.

\bibitem{Fauser2019}
{\sc D.~Fauser}, {\em Integral foliated simplicial volume and
  {$S^{1}$}-actions}, \href{https://arxiv.org/abs/1704.08538}{arXiv:1704.08538
  [math.GT]},  (2019).

\bibitem{Fauserthesis}
\leavevmode\vrule height 2pt depth -1.6pt width 23pt, {\em Parametrised
  simplicial volume and $S^1$-actions}, PhD thesis, Universität {R}egensburg,
  URN:NBN:DE:BVB:355-EPUB-404319, 2019.
\newblock \url{https://epub.uni-regensburg.de/40431/}.

\bibitem{FauserFriedlLoeh2019}
{\sc D.~Fauser, S.~Friedl, and C.~Löh}, {\em Integral approximation of
  simplicial volume of graph manifolds}, Bull. Lond. Math. Soc., 51 (2019),
  pp.~715--731.

\bibitem{FauserLoeh2019}
{\sc D.~Fauser and C.~Löh}, {\em Variations on the theme of the uniform
  boundary condition}, Journal of Topology and Analysis,  (2019).

\bibitem{frigerio-loeh-pagliantini-sauer}
{\sc R.~Frigerio, C.~L\"{o}h, C.~Pagliantini, and R.~Sauer}, {\em Integral
  foliated simplicial volume of aspherical manifolds}, Israel J. Math., 216
  (2016), pp.~707--751.

\bibitem{Goresky1978}
{\sc R.~M. Goresky}, {\em Triangulation of stratified objects}, Proc. Amer.
  Math. Soc., 72 (1978), pp.~193--200.

\bibitem{Gromov1982}
{\sc M.~Gromov}, {\em Volume and bounded cohomology}, Inst. Hautes \'{E}tudes
  Sci. Publ. Math.,  (1982), pp.~5--99 (1983).

\bibitem{Gromov1993}
\leavevmode\vrule height 2pt depth -1.6pt width 23pt, {\em Asymptotic
  invariants of infinite groups}, in Geometric group theory, {V}ol. 2
  ({S}ussex, 1991), vol.~182 of London Math. Soc. Lecture Note Ser., Cambridge
  Univ. Press, Cambridge, 1993, pp.~1--295.

\bibitem{Gromov1999}
\leavevmode\vrule height 2pt depth -1.6pt width 23pt, {\em Metric structures
  for {R}iemannian and non-{R}iemannian spaces}, vol.~152 of Progress in
  Mathematics, 1999.

\bibitem{Gromov2009}
\leavevmode\vrule height 2pt depth -1.6pt width 23pt, {\em Singularities,
  expanders and topology of maps. {I}. {H}omology versus volume in the spaces
  of cycles}, Geom. Funct. Anal., 19 (2009), pp.~743--841.

\bibitem{InoueYano1982}
{\sc H.~Inoue and K.~Yano}, {\em The {G}romov invariant of negatively curved
  manifolds}, Topology, 21 (1982), pp.~83--89.

\bibitem{Loeh2019}
{\sc C.~L\"{o}h}, {\em Cost vs. integral foliated simplicial volume},
  arXiv:1809.09660 [math.GT], to appear in Groups Geom. Dyn.,  (2019).

\bibitem{LoehPagliantini2016}
{\sc C.~L\"{o}h and C.~Pagliantini}, {\em Integral foliated simplicial volume
  of hyperbolic 3-manifolds}, Groups Geom. Dyn., 10 (2016), pp.~825--865.

\bibitem{Radeschi15}
{\sc R.~A.~E. Mendes and M.~Radeschi}, {\em A slice theorem for singular
  {R}iemannian foliations, with applications}, Trans. Amer. Math. Soc., 371
  (2019), pp.~4931--4949.

\bibitem{Moerdijk}
{\sc I.~Moerdijk and J.~Mr\v{c}un}, {\em Introduction to foliations and {L}ie
  groupoids}, vol.~91 of Cambridge Studies in Advanced Mathematics, Cambridge
  University Press, Cambridge, 2003.

\bibitem{MoerdijkPronk1999}
{\sc I.~Moerdijk and D.~A. Pronk}, {\em Simplicial cohomology of orbifolds},
  Indag. Math. (N.S.), 10 (1999), pp.~269--293.

\bibitem{Orlik}
{\sc P.~Orlik}, {\em Seifert manifolds}, Lecture Notes in Mathematics, Vol.
  291, Springer-Verlag, Berlin-New York, 1972.

\bibitem{Pansu1983}
{\sc P.~Pansu}, {\em Effondrement des vari\'{e}t\'{e}s riemanniennes, d'apr\`es
  {J}. {C}heeger et {M}. {G}romov}, no.~121-122, 1985, pp.~63--82.
\newblock Seminar Bourbaki, Vol. 1983/84.

\bibitem{Pflaum}
{\sc M.~J. Pflaum}, {\em Analytic and geometric study of stratified spaces},
  vol.~1768 of Lecture Notes in Mathematics, Springer-Verlag, Berlin, 2001.

\bibitem{Preaux2014}
{\sc J.-P. Pr\'{e}aux}, {\em A survey on {S}eifert fiber space theorem}, ISRN
  Geom.,  (2014), pp.~Art. ID 694106, 9 pages,
  \url{https://doi.org/10.1155/2014/694106}.

\bibitem{Schmidt2005}
{\sc M.~Schmidt}, {\em $L^2$-{B}etti numbers of {R}-spaces and the integral
  foliated simplicial volume}, PhD thesis, Westf\"{a}lische
  {W}ilhelms-{U}niversit\"{a}t M\"{u}nster, 2005.
\newblock \url{http://nbn-resolving.de/urn:nbn:de:hbz:6-05699458563}.

\bibitem{Sniatycki}
{\sc J.~\'{S}niatycki}, {\em Differential geometry of singular spaces and
  reduction of symmetry}, vol.~23 of New Mathematical Monographs, Cambridge
  University Press, Cambridge, 2013.

\bibitem{Sullivan1976}
{\sc D.~Sullivan}, {\em A counterexample to the periodic orbit conjecture},
  Inst. Hautes \'{E}tudes Sci. Publ. Math.,  (1976), pp.~5--14.

\bibitem{Thurston1997}
{\sc W.~P. Thurston}, {\em Three-dimensional geometry and topology. {V}ol. 1},
  vol.~35 of Princeton Mathematical Series, Princeton University Press,
  Princeton, NJ, 1997.

\bibitem{Yang1963}
{\sc C.~T. Yang}, {\em The triangulability of the orbit space of a
  differentiable transformation group}, Bull. Amer. Math. Soc., 69 (1963),
  pp.~405--408.

\bibitem{Yano1982}
{\sc K.~Yano}, {\em Gromov invariant and {$S^{1}$}-actions}, J. Fac. Sci. Univ.
  Tokyo Sect. IA Math., 29 (1982), pp.~493--501.

\end{thebibliography}

\end{document}